\documentclass[a4paper]{amsart}
\usepackage{a4}
\usepackage{amsmath}
\usepackage{amssymb}
\usepackage[applemac]{inputenc}
\usepackage{graphicx}
\usepackage{color}
\def\m{\mu}

\newcommand{\PP}{ \mathbb{P}}
\newcommand{\FF}{ \mathbb{F}}

\newcommand{\EE}{{\mathbb E}}

\newtheorem{remark}{\textbf{Remark}}[section]

\newtheorem{lemma}{\textbf{Lemma}}[section]
\newtheorem{theorem}{\textbf{Theorem}}[section]

\newtheorem{proposition}{\textbf{Proposition}}[section]

\numberwithin{equation}{section}

\newcommand\be{\begin{equation}}
\newcommand\ee{\end{equation}}
\newcommand\ba{\begin{array}}
\newcommand\ea{\end{array}}
\newcommand{\bean}{\begin{eqnarray*}}
\newcommand{\eean}{\end{eqnarray*}}
\newcommand{\NN}{\mathbb{N}}
\newcommand{\ZZ}{\mathbb{Z}}
\newcommand{\OO}{\mathcal{O}}
\def\half{\mbox{$\frac{1}{2}$}}
\def\P{\mathcal{P}}
\newcommand{\RR}{\mathbb{R}}
\def\A{\mathcal{A}}
\def\B{\mathcal{B}}
\def\C{\mathcal{C}}
\def\G{\mathcal{G}}
\def\dd{{\rm d}}
\def\eps{\varepsilon}
\def\Om{{\Omega}}

\def\F{\mathcal{F}}
\def\PP{\mathbb{P}}

\title[Discretization of some nonlinear FPK equations and applications]{On the discretization of some nonlinear Fokker-Planck-Kolmogorov  equations and applications} 

\author{
Elisabetta Carlini  \and     Francisco J. Silva   }

\thanks{  ``Sapienza'', Universit\`a di Roma, Dipartimento di Matematica Guido Castelnuovo, 00185 Rome, Italy (carlini@mat.uniroma1.it).}

\thanks{ Institut de recherche XLIM-DMI, UMR-CNRS 7252 Facult\'e des sciences et techniques 
Universit\'e de Limoges, 87060 Limoges, France (francisco.silva@unilim.fr) }

\begin{document}

\maketitle
\begin{abstract} In this work, we consider the discretization of some nonlinear Fokker-Planck-Kolmogorov equations.  The scheme we propose preserves the non-negativity of the solution, conserves the mass and, as the discretization parameters tend to zero, has limit measure-valued trajectories which are shown to solve the equation.  The main assumptions to obtain a convergence result are that the coefficients are continuous and satisfy a suitable linear growth property with respect to the space variable.  In particular,  we obtain a new proof of existence of solutions for such equations.

 We apply our results to several examples, including Mean Field Games systems and variations of the Hughes model for pedestrian dynamics. 
\end{abstract} \vspace{0.5cm}

\small {\bf AMS-Subject Classification:}  35Q84,	 65N12, 65N75.  \vspace{0.2cm}

\normalsize
\section{Introduction}  In this article we consider the nonlinear Fokker-Planck-Kolmogorov  (FPK) equation: 
$$\ba{rcl} \partial_{t} m  -\half \underset{1\leq i,j \leq d}\sum \partial_{x_i,x_j}^2\left(a_{i,j}(m,x,t)m\right)+\mbox{div}\left(b(m,x,t)m\right) &=& 0, \\[4pt]
													m(0)&=&\bar{m}_{0},
\ea \eqno(FPK)$$
where, denoting by $\P_1(\RR^d)$ (respectively $\P_2(\RR^d)$) the space of probability measures on $\RR^d$ with first (respectively second) bounded moments,  $\bar{m}_0 \in  \P_{2}(\RR^{d})$ and 
$$\ba{l}
b:C([0,T]; \P_{1}(\RR^d)) \times \RR^d \times [0,T]\to \RR^{d}, \\[5pt]
 a_{i,j}(m,x,t):= \sum_{k=1}^{r}  \sigma_{ik}(m,x,t) \sigma_{jk}(m,x,t) \; \;  \forall \; i, \; j=1, \hdots, d, \\[4pt]
\sigma_{i,j}: C([0,T]; \P_{1}(\RR^d))\times \RR^d \times [0,T]\to \RR, \hspace{0.2cm} \forall \; i=1, \hdots, d, \; \; j=1, \hdots, r. 
\ea
$$

Equation $(FPK)$ is understood as an equation for measures, in the sense that we seek for a solution $m$ in the space $C([0,T]; \P_{1}(\RR^d))$.  Note that the coefficients $b$ and $a_{i,j}$ depend, a priori,  on the values $m(t)\in \P_1(\RR^d)$ in the entire time interval $[0,T]$.  The notion of weak solution to this equation, as well as the assumptions we impose on the coefficients $b$ and $\sigma_{i,j}$, will be detailed in Section \ref{preliminaries}. 

Equation (FPK) has been mostly studied in the linear case, i.e. when $b(m,x,t)=b(x,t)$ and $\sigma_{i,j}(m,x,t)= \sigma_{i,j}(x,t)$ for all $i=1,\hdots, d$ and $j=1,\hdots, r$.  This is in part due to the close relation between solutions to $(FPK)$ and solutions to the standard Stochastic Differential Equation (SDE) 
\be\label{SDE_clasica}
\dd X(t)= b(X(t), t) \dd t + \sigma(X(t),t) \dd W(t), \hspace{0.3cm} X(0)= x,
\ee
where $\sigma$ is the matrix $d\times r$ matrix whose $(i,j)$ entry is $\sigma_{i,j}$,   $W$ is an $r$-dimensional Brownian motion and $x \in \RR^d$. Indeed, under some assumptions on  $b$ and $\sigma_{i,j}$, it is possible to show a correspondence of solutions to $(FPK)$ and the time marginal laws of weak solutions to \eqref{SDE_clasica} for almost every  $x\in \RR^d$ with respect to (w.r.t) $\bar{m}_0$ (see e.g. \cite{MR2450159,Figalli08,MR3443169} and the references therein). We refer the reader to  \cite{MR3443169}  for a systematic account of the theory of linear FPK equations and their probabilistic interpretation.  When $b(m,x,t)=b(m(t),x,t)$ and $\sigma_{i,j}(m,x,t)=\sigma_{i,j}(m(t),x,t)$ the associated FPK equation is often called McKean-Vlasov equation and several results exist concerning the well-posedness of the equation and its probabilistic interpretation (see e.g. \cite{MR762085,MR1431299}). In the case of general nonlinear coefficients, the article  \cite{MR2606196} provides an existence result  when $\sigma_{i,j}\equiv 0$  and in  the articles  \cite{MR3086740,MR3113428} sufficient conditions on the coefficients defining $(FPK)$ are given in order to ensure the existence of solutions in the second order case.  The uniqueness of solutions to $(FPK)$ is a difficult matter. The reader is referred to  \cite{MR2450159,Figalli08} for the analysis in the linear case with rough coefficients,  which borrow some ideas from \cite{MR1022305,MR2096794} dealing with the analogous problem when $\sigma_{i,j}=0$,  and to \cite{MR3399526,MR3442784,MR3522009} for the nonlinear case.  
%
%
%

Let us now comment on the  numerical approximation of FPK equations. One of the most popular numerical schemes in the linear case is  the one introduced by Chang and Cooper in \cite{ChangCooper70}.  An interesting feature of this finite difference scheme is that the discrete solution preserves some intrinsic properties of the analytical one such as non-negativity and conservation of the initial mass.   Starting from this article, several improvements have been obtained in subsequent works,  see for instance \cite {ZMV99,DM96},  where high order finite difference schemes have been proposed also for the nonlinear case. Let us also mention 
\cite{MR2966923} dealing with the application of this scheme in the context of stochastic optimal control problems.  
Finally, finite element approximations  have also been discussed  in \cite{SB93}. 

In the `70s, Kushner has provided a systematic procedure to discretize the solution of a SDE   by a discrete-time, discrete-state space Markov chain. The method the author proposes induces finite difference schemes for the associated Kolmogorov backward and forward equations (see e.g. \cite{MR0461683,MR0469468,KushnerDupuis}) and so a finite difference discretization of $(FPK)$ in the linear case. A proof of convergence of the scheme by using probabilistic tools (weak convergence of probability measures) is provided under the assumption that the  coefficients of the SDE are bounded and uniformly continuous.   More recently, in the context of Mean Field Games (MFGs) systems (see \cite{LasryLions07,HMC06}), Achdou and Capuzzo-Dolcetta introduced in \cite{AchdouCapuzzo10}   a semi-implicit finite difference scheme for  a linear FPK equation. The scheme is obtained  by computing the adjoint scheme of a  monotone and consistent discretization of the corresponding dual equation, i.e.  the  Kolmogorov backward equation. Finally, in the first order case $\sigma_{i,j}=0$, we   refer the reader to the recent articles \cite{delarue:hal-01273848,doi:10.1137/16M1084882} dealing with explicit upwind finite volume schemes  for the linear equation and to \cite{Lagoutiere_Vauchelet_book_chapter} for a similar scheme in the nonlinear and nonlocal case. Let us underline that all the schemes mentioned above share some of the good features of the Chang-Cooper scheme. Indeed, the approximated solutions  are non-negative and conserve the initial mass.  On the other hand, the main drawback of finite difference and finite element  schemes is that, when implemented in their explicit form, they have to satisfy a CFL condition, which implies a strong restriction on the size of the time steps.

A different class of methods in the linear case is the so-called path integration method, introduced in  \cite{Naess93}. These are explicit schemes where the marginal laws of the solution of \eqref{SDE_clasica} are approximated via an Euler-Maruyama discretization of \eqref{SDE_clasica} using Gaussian one step transition kernels.  Recently, in \cite{CJN15}, a convergence result for the discrete-time  marginal laws in  the $L^1$ strong topology is proved in the framework of a  linear and uniformly elliptic FPK equation with unbounded coefficients. 

Inspired by the papers \cite{MR3148086,CS15}, dealing with the approximation of Mean Field Games (MFGs), our aim in this article is to provide a discretization of the general $(FPK)$ and to establish some convergence results. In the linear case, the scheme we propose can be seen as a particular discrete-time, discrete-state space Markov chain approximation of    \eqref{SDE_clasica} and can be obtained as the dual scheme to the Semi-Lagrangian (SL) scheme proposed in \cite{CamFal95} for the associated linear  Kolmogorov backward equation. In this sense, our discretization is related to the one proposed by Kushner in \cite{MR0461683}, but using a different Markov chain approximation that allows us to avoid the CFL condition and hence consider large time steps. For this reason, we find that ``Semi-Lagrangian scheme'' is a good appellation for our discretization.   More importantly, our scheme naturally adapts to the general $(FPK)$ equation, preserves also the positivity, conserves the total mass  and allows us to obtain convergence results under  rather general assumptions on $b$ and $\sigma_{i,j}$.  Namely, in Theorem \ref{convergencia_1} we prove that local Lipschitzianity  and sublinear growth with respect to the space variable $x$, uniformly w.r.t. $m$ and $t$,  are sufficient conditions to prove that if the time step $h$  and space step $\rho$ tend to zero and satisfy that $\rho^2/h \to 0$, then every limit point of the approximated solutions (there exists at least one)  solves $(FPK)$. Under a suitable modification of the scheme,  a similar convergence result is obtained in Theorem \ref{convergencia_under_H} when the local Lipschitzianity property of $b$ and $\sigma_{i,j}$ is relaxed to merely continuity.   Naturally, if the $(FPK)$ equation admits a unique solution, then we get the convergence of the whole sequence of approximated solutions. As a by-product of this result, we obtain a new proof of existence of solutions to $(FPK)$. 

Note also that the initial condition $\bar{m}_0$ is rather general, we can consider for instance singular measures (e.g. Dirac masses) as initial distributions.   Moreover, as we will see in two nonlinear examples in Section \ref{aplications}, we can also construct our scheme by using suitable approximations of the coefficients $b$ and $\sigma_{i,j}$, in the case where such coefficients do not have an explicit form and have to be approximated,  and the convergence result remains valid.  

Let us point out that a different  SL scheme for the $(FPK)$ equation has been  proposed in \cite{KRS09} in the linear case. In this article, the advection part and the diffusion reaction term are approximated separately by using two fractional steps. Furthermore, in order to obtain a conservative scheme,  the Semi-Lagrangian method applied to the advection part needs to be adjusted. Since our scheme is derived directly from the probabilistic interpretation of $(FPK)$, it has the advantage that  the advection and diffusion terms  can be treated together and the conservation of the mass is automatically verified (see also the paper \cite{BR14}, where a conservative SL scheme for a parabolic equation in divergence form is studied).

We study in this work several applications of the scheme. We first consider two  linear equations. The first one deals with a FPK equation where the underlying dynamics models a damped noisy harmonic oscillator, while the second FPK equation is of first order and describes the distribution of a prey-predator system modeled by a Lotka-Volterra system including  effects of  seasonality. Even if these two examples are simple,  we have chosen them because of the following features. In the first model the exact solution admits an explicit expression, which allows us to quantify exactly the error of the approximation. In the second model, we consider a large time horizon in order to capture the asymptotic behavior of the system, which allows us to show the benefits of being able to chose large time steps. Next, we consider two nonlinear models. In the first one, we apply our scheme to a particular non-degenerate FPK arising in MFGs. The resulting approximation is similar to the one proposed in \cite{MR3148086,CS15}, the main difference being that the non-degeneracy of the system allows us to prove the convergence of the approximation in general dimensions. In the second model, we propose a variation of the Hughes model for pedestrian dynamics (see \cite{hughes2000flow}), where, differently from MFGs,  agents do not forecast the evolution of the crowd in order to choose their optimal trajectories. We prove an existence result for the associated FPK, as well as the convergence of the proposed discretization. 

The article is organized as follows. In Section \ref{preliminaries} we introduce the main notations and recall some fundamental results about the space $C([0,T]; \P_1(\RR^d))$, which are the keys to establish the convergence results. Section \ref{fully_discrete_scheme_section} presents the scheme, first in the linear case, for pedagogical reasons, and then in the general nonlinear case. In Section \ref{convergence_section} we prove our main results, concerning the convergence of the discretization. Finally, in Section \ref{aplications} we consider the application of the scheme to the models described in the previous paragraph.\smallskip\\

{\bf Acknowledgements:}  The first author acknowledges financial support by the Indam GNCS project ``Metodi numerici per equazioni iperboliche e cinetiche e applicazioni''.  The second author is partially supported by the ANR project
MFG ANR-16-CE40-0015-01 and the PEPS-INSMI Jeunes project ``Some open problems in Mean Field Games'' for the years 2016 and 2017.

Both authors acknowledge financial support by the PGMO project VarPDEMFG. 

\section{Preliminaries}\label{preliminaries}
We denote by $\P(\RR^d)$ the space of probability measures on $\RR^d$. Given a Borel measurable function $\Psi: \RR^{d} \to \RR^{d'}$ and $\mu \in \P(\RR^{d})$, we denote by $\Psi \sharp \mu \in \P(\RR^{d'})$ the probability measure defined as $\Psi \sharp \mu (A):= \mu(\Psi^{-1}(A))$ for all $A \in \B(\RR^{d'})$.  Given $p \in [1,\infty[$, the set $\P_{p}(\RR^{d})$ denotes the subset of $\P(\RR^d)$ with bounded $p$ moments, i.e. 
$$
\P_{p}(\RR^{d}):= \left\{ \mu \in \P(\RR^d) \; ; \; \int_{\RR^d}|x|^{p} \dd \mu(x) < \infty\right\}.
$$
Define  
$$
d_{p}(\mu_1, \mu_2):= \inf\left\{\left(\int_{\RR^{d} \times \RR^d}|x-y|^p \dd \gamma(x,y)\right)^{\frac{1}{p}} \; ;  \; \gamma \in \P(\RR^d\times \RR^d), \;  \pi_{1}\sharp \gamma= \mu_1, \; \; \pi_{2}\sharp \gamma= \mu_2  \right\},
$$
where $\pi_i: \RR^d \times \RR^d \to \RR^d$ ($i=1,2$) is defined as $\pi_{i}(x_1,x_2)= x_i$.  It is well known that $d_p$ is a distance in $\P_{p}(\RR^{d})$ (see e.g. \cite[Theorem 7.3]{Villani03}) and  that $\left(\P_{p}(\RR^{d}), d_p\right)$ is a separable complete metric space (see e.g. \cite[Proposition 7.1.5]{Ambrosiogiglisav}). Moreover, 
\be\label{inequality_dp_transport}
d_{p}(\mu_1, \mu_2)^p \leq \inf \left\{\int_{\RR^{d}} \left|x-T(x)\right|^p \dd \mu_1(x) \; ; \; T: \RR^d \to \RR^d \; \mbox{is Borel measurable and} \; \; T\sharp \mu_1=\mu_2\right\},\ee
with equality if $\mu_1$ has no atoms (see \cite[Theorem 2.1]{Ambrosio_lectures_2003}).  Finally, let us mention an important result that says that $d_1$ corresponds to the Kantorovic-Rubinstein metric, i.e. 
\be\label{distance_1_difference_lipschitz}
d_{1}(\mu_1, \mu_2)= \sup\left\{ \int_{\RR^d} f (x) \dd (\m_1-\mu_2)(x) \; ; \; f \in \mbox{Lip}_{1}(\RR^d)\right\},
\ee
where $ \mbox{Lip}_{1}(\RR^d)$ denotes the set of Lipschitz functions defined in $\RR^d$ with Lipschitz constant less or equal than $1$ (see e.g. \cite{Villani03}).  

Now, let $\mathcal{C} \subseteq C([0,T]; \P_1(\RR^{d}))$ and suppose that there exists a modulus of continuity $\bar{\omega}:[0,+\infty[ \to \RR$, i.e. $\bar{\omega}\geq 0$, $\bar{\omega}$ is continuous and $\bar{\omega}(0)=0$, such that
\be\label{equicontinuity_in_C_P1}
\sup_{\mu \in \C} d_{1}(\mu(t_1), \mu(t_2)) \leq \bar{\omega}(|t_1-t_2|) \hspace{0.4cm} \; \forall \; t_1, \; t_2 \in [0,T].
\ee
Assume in addition that there exists $C>0$ such that
\be\label{bounded_second_moments}
\sup_{\mu \in \C} \sup_{t\in [0,T]} \int_{\RR^{d}} |x|^2 \dd \mu(t)(x) \leq C.
\ee
Since the set $\left\{ \mu \in \P_1(\RR^d) \; ; \; \int_{\RR^{d}} |x|^2 \dd \mu(x) \leq C\right\}$ is compact in $\P_1(\RR^d)$ (see \cite[Proposition 7.1.5]{Ambrosiogiglisav}), \eqref{bounded_second_moments}, \eqref{equicontinuity_in_C_P1} and the Arzel\'a-Ascoli theorem yield the following result. 
\begin{lemma}\label{compactness_in_P_1} Under the above assumptions,  $\C$ is a relatively compact subset of $C([0,T]; \P_1(\RR^d))$. 
\end{lemma}

For notational convenience, for $\psi= b, \; \sigma_{i,j}$ we set  $\psi[\mu](x,t):=\psi(\mu,x,t)$.    
We say that $m \in C([0,T]; \P_1(\RR^d))$ solves $(FPK)$ if for all $t\in [0,T]$ and  $\varphi \in C^{\infty}_{0}(\RR^d)$, the space of $C^{\infty}$-functions with compact support,  we have      
\be\label{solution_FP}\ba{ll}
\int_{\RR^{d}} \varphi(x) \dd m(t)(x)=& \int_{\RR^{d}}  \varphi(x) \dd \bar{m}_{0}(x) + \int_{0}^{t}\int_{\RR^{d}} \left[   b[m](x,s) \cdot \nabla \varphi(x)   \right] \dd m(s)(x) \dd s \\[4pt]
\;  & + \int_{0}^{t}\int_{\RR^{d}} \left[  \half \sum_{i,j} a_{i,j}[m](x,s) \partial_{x_i, x_j}^{2} \varphi(x)\right] \dd m(s)(x) \dd s.
\ea\ee 
The following assumption will be the principal one in the remainder of this paper. \medskip\\
{\bf(H)} We will suppose that: \smallskip\\
{\rm(i)} The maps $b$ and $\sigma$ are continuous. \smallskip\\
{\rm(ii)} There exists $C>0$ such that  
\be\label{linear_growth}
|b[\mu](x,t)| +|\sigma[\mu](x,t)| \leq C(1+ |x|) \hspace{0.2cm} \forall \; \mu\in  C([0,T];\P_{1}(\RR^{d})), \; \;  x \in \RR^d, \; \; t\in [0,T].
\ee \smallskip

The aim of this article is to study convergent numerical schemes for  solutions to $(FPK)$ (if they exists). As it can be guessed from the references \cite{MR2450159,Figalli08,MR3443169}  in the linear case, i.e. when $b$ and $\sigma_{i,j}$ do not depend on $m$, the existence of solutions to $(FPK)$ should be related with the existence of (weak) solutions to 
the  ``extended'' McKean-Vlasov equation 
\be\label{McKean_Vlasov}
\dd X(t)= b[m](X(t), t) \dd t + \sigma[m](X(t),t) \dd W(t), \hspace{0.3cm} X(0)= X_0.
\ee
In \eqref{McKean_Vlasov}, $W$ is an  $r$-dimensional  Brownian motion defined on a probability space $(\Om, \F, \PP)$,  $m$ belongs to  $C([0,T]; \P_1(\RR^d))$ and satisfies  $m(t)= \mbox{Law}(X(t))$  for all $t\in [0,T]$, where we have denoted by $\mbox{Law}(Y)$ the law induced  in $\RR^d$ by a $d$-valued random variable $Y$, and $X_0$ is a random variable, independent of $W$,  and such that $\mbox{Law}(X_0)=m_0$. 

This observation, relating formally solutions of $(FPK)$ and \eqref{McKean_Vlasov}, leads naturally to study the laws of discrete approximations of  \eqref{McKean_Vlasov}, for which existence is not difficult to show, and then to study their limit behavior. This strategy will be followed in the next sections. 

\begin{remark}
In this article we do not tackle the study of uniqueness of solutions to $(FPK)$. As it can be seen in  \cite{MR2450159,Figalli08,MR3443169}, in the linear case, the study of uniqueness is already quite complicate in the absence of first order information, w.r.t. the space variable, of $b$ and $\sigma$.  We refer the reader to {\rm \cite{MR3399526,MR3442784,MR3522009}} for some recent and interesting results in the general nonlinear case.
%
\end{remark} 
\section{The fully-discrete scheme}\label{fully_discrete_scheme_section}
In this section we describe the scheme we propose and study its main properties.  In order to introduce the main ideas we will start by considering first the $(FPK)$ equation with $\sigma=0$ and $b$ independent of $m$, i.e. the first order linear FPK  equation, also called {\it continuity equation}. Then, we will consider the stochastic case $\sigma \neq 0$ but still with coefficients $b$ and $\sigma$ independent of $m$. Finally, the scheme for the general $(FPK)$ will easily follow by freezing the $m$ dependence of $b$ and $\sigma$. We motivate the schemes by assuming  stronger assumptions on $b$ and $\sigma$, which will imply uniqueness of solutions of the underlying SDEs,  in order to take advantage of the semi-group properties of the solutions and somehow guess a consistent approximation. 
%

%


We assume first that $\sigma\equiv 0$ and that $b$ does not depend on $m$, i.e.  $b[m](x,t)=b(x,t)$. In addition to {\bf(H)}, assume that $b$ is Lipschitz w.r.t. $x$, uniformly in $t\in [0,T]$.    For any $0 \leq s \leq t \leq T$ and $x\in \RR^d$, we set  $\Phi(x,s,t)=X(t)$ where $X$ is the unique solution of 
\be\label{McKean_Vlasov_introduction_scheme}
 \dot{X}(t')= b(X(t'), t') \; \; \; \mbox{for }t' \in ]s,T[, \hspace{0.3cm} X(s)= x.
\ee
We have that $\Phi$ defines a measurable function of $(x,s,t)$ (if $ t\leq s$ we simply set $\Phi(x,s,t)=x$). Then,  $m \in C([0,T]; \P_{1}(\RR^d))$ defined as 
\be\label{push_forward_with_flow}
m(t)(A):= \Phi(\cdot,0,t) \sharp \bar{m}_{0}(A) \hspace{0.3cm} \forall \; A \in \B(\RR^{d}), \; \; \; t\in [0,T], 
\ee
is the unique solution of $(FPK)$ (see \cite{MR2135356}). We also have that for all $t\in [0,T]$ and $h \in [0,T-t]$
\be\label{semigroup_determinista}
m(t+h)(A)= \Phi(\cdot,t, t+h) \sharp m(t)(A)  \hspace{0.3cm} \forall \; A \in \B(\RR^{d}).
\ee
Given $N \in \NN$ we set $h:=T/N$ and $t_k:=kh$ ($k=0, \hdots, N$).  Let us consider the following {\it explicit} time discretization of \eqref{push_forward_with_flow}, based on a standard explicit Euler approximation of \eqref{McKean_Vlasov_introduction_scheme} and property \eqref{semigroup_determinista}
\be\label{semi_discrete_scheme}
m_{0}=\bar{m}_0, \hspace{0.2cm} m_{k+1}:= \Phi_{k}\sharp m_{k}, \hspace{0.4cm}
\mbox{where } \; \; \; \Phi_{k}(x):= x +h b( x, t_{k}) \hspace{0.4cm} \forall \; k=0, \hdots, N-1.
\ee
The sequences $m_k$ and $\Phi_k$ ($k=0,\hdots, N$) depend of course on $h$ but we have omitted this dependence in order to ease the reading.
Let us now  introduce some standard notations that will be used for the space discretization.  Let $\rho>0$ be a given {\it space step},  and consider a uniform space  grid 
$$\mathcal{G}_{\rho}:= \{x_i= i \rho \; ; \;  i\in \ZZ^d\}.$$
Given a regular lattice $\mathcal{T}_{\rho}$ of $\RR^d$, with  vertices  belonging to $\mathcal{G}_{\rho}$, we consider a   $\mathbb{Q}_1$ basis  $(\beta_{i})_{i\in \ZZ^d}$, i.e. for all $i\in \ZZ^{d}$,  $\beta_i$  is a polynomial of degree less than or equal to $1$ and satisfies that  $\beta_i(x_j)=1$ if $i=j$ and $\beta_i(x_j)=0$, otherwise. Moreover, the support $\mbox{supp}(\beta_i)$ of $\beta_i$ is compact and  
$$0\leq \beta_i \leq 1 \; \; \;  \; \forall \; i\in \ZZ^{d}, \; \; \;   \mbox{and } \hspace{0.2cm}\sum_{i\in  \ZZ^{d}} \beta_i(x)=1 \hspace{0.3cm} \forall \; x\in \RR^d.$$
We look for a discretization of \eqref{semi_discrete_scheme} taking the form 
\be\label{discretization_dirac_masses}
m_{k}= \sum_{i \in \ZZ^{d}} m_{i,k} \delta_{x_{i}} \hspace{0.4cm} \forall \; k=0, \hdots, N-1. 
\ee
For all $i \in \ZZ^{d}$, let us define 
$$
E_{i}:= \left\{ x \in \RR^{d} \, ; \; |x-x_{i}|_{\infty} \leq \frac{\rho}{2}\right\}.
$$
In Section \ref{convergence_section} we will let $\rho \downarrow 0$, thus, without loss of generality, we can assume that $\bar{m}_0(\partial E_i)= 0$ for all $i \in \ZZ^d$. 
We define the weights $m_{i,k}$ of the Dirac masses  in \eqref{discretization_dirac_masses}   inductively as
\be\label{scheme_linear_case}
m_{i,0}=  \bar{m}_{0}(E_i),  \hspace{0.3cm}
m_{i, k+1} = \sum_{j \in \ZZ^{d}} \beta_{i} ( \Phi_{j,k})m_{j,k} \hspace{0.3cm} \forall \; k=0, \hdots, N-1, \; \; i \in \ZZ^{d},
\ee
where 
\be\label{discrete_trajectories_deterministic_case}
\Phi_{i,k}:= \Phi_k(x_i)= x_{i}+ h  b(x_{i}, t_{k}) \hspace{0.3cm} \forall \; i \in \ZZ^{d}.
\ee
The sequences of weights in \eqref{scheme_linear_case} depends on $(\rho,h)$, but, for notational convenience, we have omitted this dependence.  
\begin{remark}
{\rm(i)}
In order to understand the intuitive meaning of \eqref{scheme_linear_case}, take $d=1$, $\rho=1$ and $\beta_{i}(x):=  \max \{1 -|x-x_i|, 0\} $ for all $i\in \ZZ$, $x\in \RR$. Then, the mass $m_{i,k+1}$, at $x_i$ at time $t_{k+1}$, is obtained by first considering the set $\A_{i,k}$ of $j$'s such that $\Phi_{j,k} \in \mbox{{\rm supp}}(\beta_{i})$ and then adding the masses $m_{j,k}$ {\rm(}$j\in \A_{i,k}${\rm)} weighted by $1-  |\Phi_{j,k}-x_i|$.  For instance, if $ \Phi_{j,k}=x_{i}+1/2$ then, at the discrete time $k+1$,  half of the mass $m_{j,k}$ will be in $x_{i}$ and the other half will be in $x_{i+1}$.\smallskip\\
{\rm(ii)} In this deterministic setting if $d=1$ it is easy to check that \eqref{scheme_linear_case} coincides with the scheme proposed in \cite{MR2771664}.
\end{remark}
Now, if   $\sigma[m](x,t)=\sigma(x,t)$ is not identically zero we can consider the same type of scheme, taking into account that the characteristics curves are stochastic. Indeed, consider a filtered probability space $(\Om, \F, \FF ,\PP)$, an $r$-dimensional Brownian motion $W$ defined in this probability space and adapted to the filtration $\FF:= \{\F_{t}\}_{t\in [0,T]}$. Define $\Phi: \Omega \times \RR^{d} \times [0,T] \times [0,T] \to \RR^d $ as $\Phi(\omega,x,s,t)=x$ if $t\leq s$ and, for $s<t$, $\Phi(\omega,x,s,t)=X(t,\omega)$, where $X$ solves
\be\label{McKean_Vlasov_introduction_scheme_stochastic}
 \dd X(t')= b(X(t'), t')\dd t'+ \sigma(X(t'),t') \dd W(t') \hspace{0.3cm} \mbox{for } \, t' \in ]s,T[, \; \; \,  X(s)= x.
\ee
Then, assuming that $b$ and $\sigma$ are Lipschitz with respect to   $x$, uniformly in $t\in [0,T]$, we have that (see e.g. \cite{Figalli08})
\be\label{push_forward_with_flow_stochastic_in_mean}
m(t)(A):= \int_{\Om} \Phi(\omega, \cdot, 0,t) \sharp \bar{m}_{0}(A)\dd \PP(\omega)= \EE\left(\Phi(\cdot,0,t) \sharp \bar{m}_{0}(A) \right)\hspace{0.3cm} \forall \; A \in \B(\RR^{d}), \; \; \; t\in [0,T],
\ee
where, as usual, we have omitted the dependence of $\Phi$ on $\omega$ inside the expectation.  Analogously to \eqref{semigroup_determinista}, we have that 
\be\label{semigroup_estocastico}
m(t+h)(A)= \int_{\Om} \Phi(\omega, \cdot, t,t+h) \sharp m(t)(A)\dd \PP(\omega)=\EE\left(\Phi(\cdot,t, t+h) \sharp m(t)(A)\right)  \hspace{0.3cm} \forall \; A \in  \B(\RR^{d}).
\ee
Therefore, if we discretize the Brownian motion $W$ by an $r$-dimensional random walk with $N$ time steps,  the stochastic characteristic 
$$
X(t+h)= X(t)+ \int_{t}^{t+h} b(X(t'),t') \dd t' + \int_{t}^{t+h}\sigma(X(t'),t') \dd W(t'),
$$
can be approximated with an explicit Euler scheme by
\be\label{xtplush}
X(t+h)= X(t)+ h b(X(t),t)+ \sqrt{rh} \sigma(X(t),t)Z, 
\ee
where $Z$ is an $r$-valued random variable, independent of $X(t)$,  satisfying that for all $\ell=1, \hdots, r$, 
\be\label{probabilities_random_walk}
\PP( \{ Z^{\ell}=1)= \PP(Z^{\ell}=-1)= \frac{1}{2d} \hspace{0.5cm} \mbox{and } \hspace{0.5cm} \PP\left( \bigcup_{1 \leq \ell_1 < \ell_2 \leq r} \{ Z^{\ell_1} \neq 0 \} \cap  \{ Z^{\ell_2} \neq 0 \} \right)=0. 
\ee
Relations \eqref{xtplush}-\eqref{probabilities_random_walk} motivate the following extensions of $\Phi_{i,k}$, defined in  \eqref{discrete_trajectories_deterministic_case}, 
\be\label{discrete_trajectories_stochastic_case}
\ba{rcl}
\Phi_{i,k}^{\ell,+} &:=& x_{i} + h b(x_{i},t_{k})+ \sqrt{r h }\sigma_{\ell}(x_{i},t_{k}) \; \; \; \forall \; i\in \ZZ^{d}, \; \; k=0, \hdots, N-1, \; \; \ell=1,\hdots,r, \\[4pt]
\Phi_{i,k}^{\ell,-} &: =& x_{i} + h b(x_{i},t_{k})-\sqrt{r h }\sigma_{\ell}(x_{i},t_{k}) \; \; \; \forall \; i\in \ZZ^{d}, \; \; k=0, \hdots, N-1, \; \; \ell=1,\hdots,r.
\ea
\ee
Inspired by \eqref{scheme_linear_case}, relation \eqref{semigroup_estocastico} induces the following {\it explicit} scheme
\be\label{scheme_linear_stochastic_case}\ba{rcl}
m_{i,0}&:=& \bar{m}_{0}(E_i) \hspace{0.4cm} \forall \; i\in \ZZ^d,  \\[6pt]
m_{i, k+1}&:=& \frac{1}{2r}\sum\limits_{\ell=1}^{r}\sum\limits_{j \in \ZZ^{d}} \left[\beta_{i} ( \Phi_{j,k}^{\ell,+})+\beta_{i}(\Phi_{j,k}^{\ell,-})\right] m_{j,k}   \hspace{0.4cm} \forall \; i\in \ZZ^d, \; \; k=0,\hdots, N-1. 
\ea \ee
\begin{remark}\label{conservation_property} Note that the previous scheme is conservative. Indeed, for all $k=0,\hdots, N-1$, 
$$
\sum_{i\in \ZZ^d} m_{i,k+1} =\sum\limits_{j \in \ZZ^{d}} \frac{m_{j,k}}{2r}\sum\limits_{\ell=1}^{r} \sum_{i\in \ZZ^d}  \left[\beta_{i} ( \Phi_{j,k}^{\ell,+})+\beta_{i}(\Phi_{j,k}^{\ell,-})\right] =\sum_{i\in \ZZ^d} m_{i,k},
$$
and so $\sum_{i\in \ZZ^d} m_{i,k+1}=\sum_{i\in \ZZ^d} m_{i,0}=1$. 
\end{remark}
{\it Markov chain interpretation:} Note that \eqref{scheme_linear_stochastic_case} can be interpreted in terms of a  discrete-time and countably-state space Markov chain. Indeed, given  the initial law $m_{\cdot,0}$ on $\G_{\rho}$, consider the non-homogeneous Markov chain $\{X_{k} \; ; \; k=0, \hdots, N\}$ with values in $\G_{\rho}$ defined by the previous initial law and the transition probabilities
$$
p_{ji}^{(k)}:= \PP\left( X_{k+1}=x_{i} \; \big| \; X_{k}=x_{j}\right):= \frac{1}{2r}\sum\limits_{\ell=1}^{r}\left[\beta_{i} ( \Phi_{j,k}^{\ell,+})+\beta_{i}(\Phi_{j,k}^{\ell,-})\right] \hspace{0.3cm} \; \forall \; i, j \in \ZZ^{d}, \; \; k=0,\hdots,N-1.
$$
Then, \eqref{scheme_linear_stochastic_case} gives the distribution of $X_k$ for all $k=0, \hdots, N$. 
\begin{remark} {\rm(i)} Note that if $\sigma\equiv 0$, we recover the scheme \eqref{scheme_linear_case}. \smallskip\\
{\rm(ii)} As we will see in Section \ref{convergence_section}, the Markov chain $(X_k)_{k=0}^N$ is a consistent approximation, in the sense of Kushner {\rm(}see \cite{MR0469468}{\rm)}, of the diffusion in \eqref{McKean_Vlasov_introduction_scheme_stochastic} with $s=0$ and with $\mbox{{\rm Law}}(X_0)= \bar{m}_0$. It is easily seen that, as a function of $\bar{m}_0$,  scheme \eqref{scheme_linear_stochastic_case} can be formally understood as  the dual scheme associated to the Semi-Lagrangian scheme {\rm(}see \cite{MR1906636}{\rm)} for the Kolmogorov backward equation 
$$\ba{rcl} \partial_{t} u  -\half \underset{1\leq i,j \leq d}\sum  a_{i,j} \partial_{x_i,x_j}^2u+b\cdot \nabla u&=& 0, \\[4pt]
													u(\cdot, T)&=&g(\cdot),
\ea $$
as a function of $g \in C_{b}(\RR^{d})$ {\rm(}where $C_{b}(\RR^{d})$ is the space of bounded continuous functions in $\RR^d${\rm)}.  \smallskip\\
{\rm(iii)} In  \cite[Section 3.1]{carlini2016PREPRINT}, it is shown that scheme \eqref{scheme_linear_stochastic_case} can also be constructed  from the weak formulation of $(FPK)$ {\rm(}when $b$ and $\sigma$ are independent of $m${\rm)}. 
\end{remark}

In the general non-linear case, as we have explained at the end of Section \ref{preliminaries},  formally, $m$ solves $(FPK)$ iff  for all $t\in [0,T]$, we have that $m(t)= \mbox{Law}(X(t))$, where $X$ solves \eqref{McKean_Vlasov} (assuming that \eqref{McKean_Vlasov} admits a solution in a weak  sense). On the other hand, even in the particular case of regular coefficients and  local in time dependence on $m$, i.e. $b[m](x,t)=b(m(t),x,t)$ and 
$\sigma[m](x,t)=\sigma(m(t),x,t)$, with $b$ and $\sigma$ regular w.r.t. $x$,  we have that $X$  is not a Markov process. Nevertheless, loosely speaking again, $X$  solves \eqref{McKean_Vlasov} iff $\mbox{{\rm Law}}(X(\cdot)) \in C([0,T]; \P_1(\RR^{d}))$ is a fixed point of the application 
\be\label{fixed_point_function}
\mu \in C([0,T]; \P_1(\RR^{d})) \mapsto \F(\mu):=\mbox{{\rm Law}}(X[\mu](\cdot)) \in C([0,T]; \P_1(\RR^{d})), 
\ee
where $X[\mu](\cdot)$ solves
\be\label{McKean_Vlasov_introduction_scheme_stochastic_nonlinear_fixed_point_sde}
 \dd X(t)= b[\mu](X(t), t)\dd t + \sigma[\mu](X(t),t) \dd W(t) \; \; \mbox{for } t\in ]0,T[, \hspace{0.3cm} X(0)= X_0.
\ee
Since for every fixed $\mu$, $X[\mu]$ defines a  Markov diffusion, we can apply \eqref{scheme_linear_stochastic_case} to approximate its law. 

Even if the previous discussion is purely formal,  it provides the idea to construct a natural discretization of $(FPK)$ by considering a discrete version of the fixed-point problem \eqref{fixed_point_function}, which will be constructed using  \eqref{scheme_linear_stochastic_case}. However, since $b[\cdot](x,t)$ and $\sigma[\cdot](x,t)$ act on $C([0,T]; \P_1(\RR^{d}))$, given $\rho$ and $h$ we first need to extend elements on \small
$$\mathcal{S}^{\rho,h}:= \left\{ (\mu_{i,k})_{i\in \ZZ^d, \; k=0,\hdots, N} \; ; \; \mu_{i,k} \geq 0,  \; \; \forall \; i\in \ZZ^d,   \; \; \sum_{i \in \ZZ^d} \mu_{i,k} =1, \; \;  \sum_{i \in \ZZ^d} |x_i| \mu_{i,k}<\infty, \; \;  \mbox{for all } k=0, \hdots, N\right\},$$ \normalsize
to elements in $C([0,T]; \P_1(\RR^{d}))$. This can be naturally done  by using time interpolation. Given  $\mu \in \mathcal{S}^{\rho,h}$, we still denote by  $\mu$ the element of $C([0,T]; \P_{1}(\RR^{d}))$ defined by
\be\label{time_interpolation}
\mu(t):= \left(\frac{t-t_{k}}{h}\right)\sum_{i\in \ZZ^{d}} \mu_{i,k+1} \delta_{x_{i}} +  \left(\frac{t_{k+1}-t}{h}\right)\sum_{i \in \ZZ^d} \mu_{i,k}\delta_{x_{i}} \hspace{0.4cm} \mbox{if $t\in [t_{k}, t_{k+1}[$,}
\ee
for all $k=0, \hdots, N-1$. Using this notation, define 
\be\label{fixed_point_function_discrete}
\mu \in \mathcal{S}^{\rho,h} \mapsto \F^{\rho,h}(\mu):=\left(\mbox{{\rm Law}}(X_k[\mu])\right)_{k=0,\hdots, N} \in \mathcal{S}^{\rho,h}, 
\ee
where  we compute $\PP(X_{k}[\mu]=x_i):=m_{i,k}[\mu]$ recursively with   \eqref{scheme_linear_stochastic_case} with $\Phi_{j,k}^{\ell,+}$ and $\Phi_{j,k}^{\ell,-}$ replaced by 
$$
\Phi_{i,k}^{\ell,+}[\mu]  :=  x_{i} + h b[\mu](x_{i},t_{k})+ \sqrt{r h }\sigma_{\ell}[\mu](x_{i},t_{k}), \; \; \; 
\Phi_{i,k}^{\ell,-}[\mu]  : =  x_{i} + h b[\mu](x_{i},t_{k})-\sqrt{r h }\sigma_{\ell}[\mu](x_{i},t_{k}),
$$
respectively. For $\mu \in \mathcal{S}^{\rho,h}$ let us set $\nu_k[\mu]:=\left(\F^{\rho,h}(\mu)\right)_k$ ($k=0,\hdots,N$). By definition of the scheme, using that $\bar{m}_0 \in \P_2(\RR^d)$ and that $\bar{m}_0(\partial E_i)= 0$, we have 
$$
\int_{\RR^d} |x|^2 \dd \nu_0[\mu](x)= \sum_{i \in \ZZ^d} |x_i|^2 \bar{m}_{0}(E_i)= \sum_{i \in \ZZ^d} \int_{E_i} |x - (x-x_i)|^2 \dd \bar{m}_0(x)\leq 2 \int_{\RR^d} |x|^2 \dd \bar{m}_0(x) + d \rho^2/2 <+\infty.
$$
Moreover, arguing exactly as in the proof of Proposition \ref{second_moments_bounded} in the next section, under ${\bf(H)}{\rm(ii)}$ we obtain the existence of $c>0$, independent of $\mu$, such that 
\be\label{uniform_bound_second_moments_Sh}
\int_{\RR^{d}} |x|^{2} \dd \nu_{k}[\mu](x)= \sum_{i\in \ZZ^{d}} |x_i|^{2} \nu_{i,k} [\mu]\leq c \hspace{0.3cm} \forall \; k=0,\hdots, N, \; \; \; \forall \; \mu \in \SS^{\rho,h}.
\ee
 In particular, $\F^{\rho,h}$ is well-defined.  The discretization of $(FPK)$ we propose is
\be\label{scheme_nonlinear_stochastic_case_I}
\mbox{find $m  \in \SS^{\rho,h}$ such that} \hspace{0.2cm} m=\F^{\rho,h}(m),
\ee
or equivalently,  find $m\in \SS^{\rho,h}$ such that
\be\label{scheme_nonlinear_stochastic_case_II}\ba{l}
m_{i,0} = \bar{m}_{0} (E_{i})  \hspace{0.4cm} \forall \; i\in \ZZ^d,  \\[6pt]
m_{i, k+1} = \frac{1}{2r}\sum\limits_{\ell=1}^{r}\sum\limits_{j \in \ZZ^{d}} \left[\beta_{i} ( \Phi_{j,k}^{\ell,+}[m])+\beta_{i}(\Phi_{j,k}^{\ell,-}[m])\right] m_{j,k} \hspace{0.4cm} \forall \; i\in \ZZ^d, \; \; k=0,\hdots, N-1. 
\ea \ee 
Now, let us prove the existence of solutions of \eqref{scheme_nonlinear_stochastic_case_I}. In the following proof we identify $\SS^{\rho,h}$ with a subset of $\P_1(\RR^{d})^{N+1}$ by letting for all $\mu \in \SS^{\rho,h}$
\be\label{identification_Sh_P}
\mu_{k}= \sum_{i\in  \ZZ^{d}} \mu_{i,k} \delta_{x_i} \hspace{0.3cm}  \forall \; k=0,\hdots, N.
\ee

\begin{proposition}\label{existence_by_fixed_point} There exists at least one solution $m^{\rho,h} \in \mathcal{S}^{\rho,h}$ of \eqref{scheme_nonlinear_stochastic_case_I}. 
\end{proposition}
\begin{proof}  As before, for $\mu \in \SS^{\rho,h}$ denote by  $\nu_k[\mu]:=\left(\F^{\rho,h}(\mu)\right)_k$ ($k=0,\hdots,N$). Let $c>0$ be such that \eqref{uniform_bound_second_moments_Sh} holds. 
%
Then, defining 
$$\SS^{\rho,h}_{c}:= \left\{ \mu \in \SS^{\rho,h} \; ; \; \sum_{i\in \ZZ^{d}} |x_{i}|^{2} \mu_{i,k}\leq c, \hspace{0.4cm} \; \; \forall \; k=0,\hdots, N\right\},$$
we have that $\SS^{\rho,h}_{c}$ is convex and $\F^{\rho,h}(\SS^{\rho,h}_{c})\subseteq \SS^{\rho,h}_{c}$. Moreover, by  \cite[Proposition 7.1.5 and Proposition 5.1.8]{Ambrosiogiglisav}, Fatou's Lemma and the identification \eqref{identification_Sh_P}, we have that  $\SS^{\rho,h}_{c}$ is a compact subset of  $\P_1(\RR^{d})^{N+1}$. Finally,   if $\mu_n \in  \SS^{\rho,h}$ converge to $\mu \in  \SS^{\rho,h}$, seen as elements of $\P_1(\RR^{d})^{N+1}$,  then, using the extension \eqref{time_interpolation}, assumption {\bf(H)}{\rm(i)} implies that $ \Phi_{j,k}^{\ell,+}[\mu_n]$ and $ \Phi_{j,k}^{\ell,-}[\mu_n]$ converge to  $ \Phi_{j,k}^{\ell,+}[\mu]$ and $ \Phi_{j,k}^{\ell,-}[\mu]$, respectively, which implies the continuity of  $\F^{\rho,h}$. Since the topology of $\P_1(\RR^d)$  is the restriction to $\P_1(\RR^d)$ of the topology induced by the modified Kantorovic-Rubinstein norm on the linear  space of all bounded Borel measures on $\RR^d$ with respect to which all the Lipschitz functions are integrable (see the discussion before Proposition 1.1.4 in \cite{MR3058744}),  the existence of a solution of \eqref{scheme_nonlinear_stochastic_case_I} follows from  Schauder's fixed point theorem.  
\end{proof}
The computation in Remark \ref{conservation_property} applies  in the nonlinear case and so the scheme is conservative. 
\begin{remark}\label{explication_explicito_implicito}{\rm[Explicit and implicit schemes]} 
Note that if for all $t\in [0,T]$,  $b[m](x,t)=\hat{b}(m(\cdot \wedge t), x,t)$ and $\sigma[m](x,t)=\hat{\sigma }(m(\cdot \wedge t), x,t)$ for some functions $\hat{b}$ and $\hat{\sigma}$ defined in $C([0,T]; \P_{1}(\RR^{d}))\times \RR^{d} \times [0,T]$, then the scheme \eqref{scheme_nonlinear_stochastic_case_II} is explicit in the time steps and the existence of solution, as well as the uniqueness, of the scheme is straightforward. In the general case, the scheme is implicit in the time steps and, as we have seen in the proof of the previous proposition, the existence of solutions is a consequence of Shauder fixed point theorem. The latter situation is the one we face when we consider MFGs, as we will see in Section \ref{subsection_mfg}. In the implicit cases, the uniqueness of solutions is generally not true and its fulfilment depends on the problem at hand.   
\end{remark}

\section{Convergence analysis}\label{convergence_section}
In this section we prove our main results concerning the convergence of solutions to \eqref{scheme_nonlinear_stochastic_case_II} to solutions to $(FPK)$. In our first main result in Theorem \ref{convergencia_1}, we prove the desired convergence result under an additional local Lipschitz assumption on $b$ and $\sigma$, with respect to the space variable, and  suitable conditions on the time and space steps. In Theorem \ref{convergencia_under_H}, we consider a variation of the scheme in Section \ref{fully_discrete_scheme_section}, with regularized coefficients, and we prove a similar convergence result by assuming only   {\bf(H)} and some conditions on the discretization parameters.

Let us first introduce and recall some classical properties of the linear interpolation operator we consider (see e.g. \cite{Ciarlet,quartesaccosaleri07} for further details).  Let $B(\mathcal{G}_{\rho})$ the space of bounded functions on $\mathcal{G}_{\rho}$ and for $f\in B(\G_\rho)$ set $ f_{i}:= f(x_{i})$. We consider the following  linear interpolation operator 
\be\label{definterpolation}
I[f](\cdot):=\sum_{i\in  \ZZ^{d}} f_i\beta_i(\cdot) \hspace{0.3cm} \mbox{for } f \in B(\mathcal{G}_{\rho}).
\ee
Given $\phi \in C_{b}(\RR^{d})$, let us define $\hat{\phi} \in B(\mathcal{G}_{\rho})$ by $\hat{\phi}_i:=\phi(x_i)$ for all $i \in \ZZ^{d}$. 
 Suppose that $\phi: \RR^{d} \to \RR$ is Lipschitz with constant $L$. Then, 
\be\label{lipchitzianidad} I[\hat{\phi}] \; \;\; \mbox{is Lipschitz with constant $\sqrt{d} L \; $ and } \; \sup_{x\in \RR^{d}}| I[\hat{\phi}](x)-\phi(x)|=c_0 \rho,\ee 
for some $c_0>0$.  On the other hand,
if $\phi \in \C^{2}(\RR^{d})$, with bounded second derivatives, then there exists $c_1>0$ such that 
\begin{equation}\label{www2}
\sup_{x\in \RR^{d}}| I[\hat{\phi}](x)-\phi(x)|=c_1 \rho^2.
\end{equation}
Now, let $\{N_{n}\}_{n\in \NN}$ be a sequence in $\NN$ such that $N_n \to \infty$ as $n\to \infty$ and set $h_n:= T/N_n$. Given a sequence of space steps $\rho_n$, such that $\rho_n \to 0$ as $n\to \infty$, we want to study the limit behavior of  the extensions to $C([0,T];\P_1(\RR^{d}))$, defined in \eqref{time_interpolation}, of sequences of solutions   $m^{n}:=m^{\rho_n,h_n}\in \SS^{\rho_n,h_n}$ of   \eqref{scheme_nonlinear_stochastic_case_I}, with $\rho=\rho_n$ and $h=h_{n}$ (by Proposition \ref{existence_by_fixed_point}  we now that    \eqref{scheme_nonlinear_stochastic_case_I}  admits at least one solution). 

First note that by considering the transport plan $T(x)=x_i$ if $x\in E_i$, and arbitrarily defined in $\partial E_i$ (because $\bar{m}_0(\partial E_i)= 0$),  we have that $T\sharp \bar{m}_0=m^n(0)$. Thus,   inequality \eqref{inequality_dp_transport} with $p=1$ yields
\be\label{approximation_initial_condition}
d_{1}\left(\bar{m}_0, m^{n}(0)\right) \leq \int_{\RR^{d}}| x-T(x)|\dd \bar{m}_0(x) =\sum_{i\in \ZZ^d}\int_{E_i}|x-x_i| \dd \bar{m}_0(x) \leq \sqrt{d} \rho_n/2,
\ee
which implies that $m^n(0) \to \bar{m}_0$ in $\P_1(\RR^d)$ as $n\to \infty$.  We  prove in this section  that under suitable conditions over $\rho_n$ and $h_n$  the set $\C:=\{ m^n \; ; \; n\in \NN\}$
satisfies \eqref{equicontinuity_in_C_P1} and \eqref{bounded_second_moments}. Therefore,   Lemma \ref{compactness_in_P_1} will imply that $m^n$ has at least one limit point $m \in C([0,T]; \P_1(\RR^{d}))$.   In the proof of  \eqref{equicontinuity_in_C_P1} and \eqref{bounded_second_moments}  we will need some properties of the Markov chain $X^{n}$,  defined by the transition probabilities 
$$
p_{ji}^{n,k}:= \PP\left( X_{k+1}^{n}=x_{i} \; \big| \; X_{k}^{n}=x_{j}\right):= \frac{1}{2r}\sum\limits_{\ell=1}^{r}\left[\beta_{i} ( \Phi_{j,k}^{\ell,+}[m^n])+\beta_{i}(\Phi_{j,k}^{\ell,-}[m^n])\right] \hspace{0.3cm} \; \forall \; i, j \in \ZZ^{d}, 
$$
and $ k=0,\hdots,N-1$.  Note that \eqref{scheme_nonlinear_stochastic_case_II} implies that the mariginal distributions of this chain are given by $m^n$. 
%
Moreover, it is easy to check  that \eqref{lipchitzianidad} (resp. \eqref{www2}) implies that if $\phi: \RR^{d} \to \RR$ is Lipschitz (reps. $C^2$ with bounded second derivatives), then 
\be\label{est.esp}\ba{l} \EE\left( \phi(X_{k+1}^n) \big| X_k^n=x_i\right)= \frac{1}{2r}\sum_{\ell=1}^{r}\left[ I[\hat{\phi}](\Phi^{\ell,+}_{i,k}[m^n])+I[\hat{\phi}](\Phi^{\ell,-}_{i,k}[m^n])\right]  \\[6pt]
\hspace{3.4cm}=\frac{1}{2r} \sum_{\ell=1}^{r}\left[ \phi( \Phi^{\ell,+}_{i,k}[m^n])+\phi(\Phi^{\ell,-}_{i,k}[m^n])\right] +  O(\rho_n), \\[6pt]
\mbox{(resp.) } \; \; \EE\left( \phi(X_{k+1}^n) \big| X_k^n=x_i\right)= \frac{1}{2r} \sum_{\ell=1}^{r}\left[ \phi( \Phi^{\ell,+}_{i,k}[m^n])+\phi(\Phi^{\ell,-}_{i,k}[m^n])\right]+  O(\rho_n^2).
\ea
\ee
\begin{proposition}\label{second_moments_bounded}
Suppose that $\rho_{n}^2=O(h_n)$. Then, there exists  a constant  $c>0$ such that
\be\label{secmom} \sup_{n\in \NN} \sup_{t\in [0,T]}\int_{\RR^{d}} |x|^{2}  \dd m^n(t) \leq  c.\ee

\end{proposition}
\begin{proof} By \eqref{time_interpolation}, it is enough to show that there exists $c>0$, independent of $n$, such that 
\be\label{sec_mom_on_discrete_time} \sup_{k=0, \hdots,N_{n}}\int_{\RR^{d}} |x|^{2}  \dd m^n(t_k) \leq  c.\ee
For notational convenience we will omit the superscript $n$. By definition,  
$$\int_{\RR^{d}} |x|^{2} \dd m(t_{k+1})(x) = \sum_{i \in \ZZ^{d}}|x_{i}|^{2}m_{i,k+1}= \EE(|X_{k+1}|^{2}), $$
from which, using \eqref{est.esp} and {\bf(H)}{\rm(ii)}, 
$$\ba{rcl}
\EE(|X_{k+1}|^{2}) &=&  \sum_{i \in \ZZ^{d}} \EE\left(|X_{k+1}|^{2}\big| X_{k}=x_i\right)m_{i,k}, \\[4pt]
\;  & = &\frac{1}{2r} \sum_{\ell=1}^{r} \sum_{i \in \ZZ^{d}}\left[ \big|\Phi_{i,k}^{\ell,+}[m]\big|^{2}+  \big|\Phi_{i,k}^{\ell,-}[m]\big|^{2}\right]m_{i,k}+O(\rho_{n}^{2}), \\[6pt]
\; &=& \EE\left[ \big|X_k+ h_n b[m](X_k,t_k)+ \sqrt{rh_n} \sigma[m](X_k,t_k)Z_{k}\big|^2  \right]+O(\rho_{n}^{2}) , \\[6pt]
\; &= & \EE\left[ |X_k|^{2} + h_n^{2}|b[m](X_k,t_k)|^{2} + rh_n\sum_{\ell=1}^{r}|\sigma_{\ell}[m](X_k,t_{k})|^{2}+2h_n X_k \cdot b[m](X_k,t_k)\right]\\[6pt]
\; & \; &+ O(\rho_{n}^{2}),\\[6pt]
\; &\leq & (1+Ch_n)\EE(|X_k|^{2})+  Ch_n +  O(\rho_n^{2}),
\ea
$$
where $Z_k$ is an $r$-valued random variable, independent of $X_k$,   satisfying \eqref{probabilities_random_walk} and
$C$ is independent of $n$. Iterating, we get
$$\ba{rcl}\int_{\RR^{d}} |x|^{2} \dd m(t_{k+1})(x)&\leq& (1+Ch_n)^{\frac{T}{h_n}} \EE(|X_0|^{2})+(Ch_{n}+O(\rho_{n}^2)) \sum_{k=0}^{N}(1+Ch_n)^{k} + O(\rho_{n}^{2})  \\[6pt]
\; & \leq & e^{CT}  \EE(|X_0|^{2})+  \frac{T}{h_n}\left[Ch_n+O(\rho_{n}^2)\right] e^{CT} +  O(\rho_{n}^2) \\[6pt]
&\leq& e^{CT}  \EE(|X_0|^{2})+ \left(CT+O\left(\frac{\rho_{n}^2}{h_{n}}\right)\right) e^{CT},
\ea$$
from which the result follows.
\end{proof} \medskip

Now, we prove a consistency property of the chain $X^n$ in the spirit of Kushner \cite{MR0469468}. 
For all $0\leq k\leq N_n-1$ let us define $\delta_{k} X^n:= X_{k+1}^n-X_k^n$, $Y_k^n:=\delta_{k} X^n- \EE\left(\delta_{k} X^n|X_k^n \right)$.
\begin{lemma}\label{consistenciacadena} For all $k=0,\hdots,N_n-1$ we have that
$$\ba{rcl} \EE(\delta_k X^n|X_k^n)&=&h_n b[m^n](X_k^n,t_k), \\[6pt]
  \EE( |Y_k^n|^2|X_k^n)&=& h_n \sum_{\ell=1}^{r}  |\sigma_{\ell}[m^n](X_k^n,t_k)|^2 + O(\rho_n^2).\ea$$
\end{lemma}
\begin{proof}
By definition  of $p_{i_{k},i_{k+1}}^{n,k}$ we have
$$\ba{rcl}
\EE(\delta_k X^n|X_k^n=x_{i_k})&=& \sum_{i_{k+1}}\left(x_{i_{k+1}}- x_{i_{k}}\right)p_{i_{k},i_{k+1}}^{n,k}\\[6pt]
\; &=& \frac{1}{2r}\sum_{\ell=1}^{r}\left( I[\mbox{id}-x_{i_k}](\Phi^{\ell,+}_{i_{k},k}[m^n])+ I[\mbox{id}-x_{i_k}](\Phi^{\ell,-}_{i_{k},k}[m^n])\right)\\[6pt]
\; &=& h_n b[m^n](x_{i_k},t_k),
\ea
$$ 
where $\mbox{id}(x)=x$ and the last equality follows from the fact that $I[\mbox{id}-x_{i_k}](y)= y-x_{i_k}$ for all $y\in \RR^d$. 
Analogously,
$$ \EE( |Y_k^n|^2|X_k^n=x_{i_k})= \sum_{i_{k+1}}\left[x_{i_{k+1}}- x_{i_{k}}- \EE(\delta_k X^n|X_k^n=x_{i_k})\right]^2 p_{i_{k},i_{k+1}}^{n,k}.$$
Using \eqref{www2} and the definition of  $p_{i_{k},i_{k+1}}^{n,k}$ again we get that 
$$
 \EE( |Y_k^n|^2|X_k=x_{i_k})= h_n  \sum_{\ell=1}^{d}|\sigma_{\ell}[m^n](x_{i_k},t_k)|^2 + O(\rho_n^2),
$$
from which the result follows.
\end{proof}
Now, we prove that  $\C:= \{m^n \; ; \; n\in \NN\}$ satisfies \eqref{equicontinuity_in_C_P1}.
\begin{proposition} \label{masmammadapppqpqwppa} 
Suppose that $\rho_n^2=O(h_n)$. Then, there exists  a constant  $C>0$ such that
\be\label{equicontinuity_in_d2} \sup_{n\in \NN} d_{2}(m^n(t),m^n(s) )\leq C|t-s|^{\half} \hspace{0.3cm} \forall \;  t, s \in [0,T].\ee
In particular, since $d_1 \leq d_2$, we have that $\C$ satisfies \eqref{equicontinuity_in_C_P1}.
\end{proposition}

\begin{proof} The proof is divided into two steps: \\
{\it Step 1:} We first show that  for given $N_n$ there exists a constant $C$, independent of $n$, such that
\be\label{equicontinuity_discrete_times}d_{2}(m^n(t_k),m^n(t_k') )\leq C\sqrt{|k-k'|h_n}  \hspace{0.4cm} \; \forall \; k, k'=0, \hdots, N_n.\ee
We assume, without loss of generality, that $k'=0$. For notational convenience, we omit the superscript $n$ on the sequences $X_k^n$, $\delta_k X^n$ and $  Y_k^n$. 
By the definition of $d_2$ we have
\be\label{estimated2}
d_2(m^n(t_k),m_0^n)\leq \left[\EE(|X_k-X_0|^2)\right]^{\half}.
\ee
We have that 
\be\label{decomposicion}
\EE\left(|X_k-X_0|^2\right)= \EE  \left|\sum_{p=0}^{k-1} \left( Y_p+ \EE\left(\delta_{p} X|X_p\right) \right)\right|^2 \leq 2 \EE  \left|\sum_{p=0}^{k-1}Y_p\right|^2 + 2\EE  \left|\sum_{p=0}^{k-1}\EE\left(\delta_{p} X|X_p \right)\right|^2.
\ee
Now, for  $ 0 \leq r < l \leq k-1$ conditioning on $\F_{l}:= \sigma (X_0, \hdots, X_{l})$ and using that, by the Markov property, $\EE(\delta_{l} X| \F_{l})=\EE(\delta_{l} X| X_l)$ we get
$$ \EE(Y_l \cdot Y_{r})= \EE\left[\left(\delta_{l} X- \EE\left(\delta_{l} X|X_l \right) \right)\cdot Y_r\right]=\EE( \EE\left[\left(\delta_{l} X- \EE\left(\delta_{l} X|X_l\right) \right)\big| \F_{l} \right]\cdot Y_r)=0,   $$
and so, by Lemma \ref{consistenciacadena},
\be\label{ypcuadri}
\ba{rl}
\EE  \left|\sum_{p=0}^{k-1}Y_p\right|^2 &=  \sum_{p=0}^{k-1}\EE( \EE(|Y_p|^2|X_p)) \\[6pt]
\; &=h_n \sum_{p=0}^{k-1}\sum_{\ell=1}^{r}\EE(|\sigma_{\ell}[m^n](X_p,t_p)|^2) + O\left(k \rho_n^2\right)\\[6pt]
\; &\leq Ch_nk\left(1+ \sup_{p=0,\hdots, N} \EE |X_p|^2\right) +O\left(k \rho_n^2\right). \ea
\ee
On the other hand, using  Lemma \ref{consistenciacadena} again,
$$\left|\sum_{p=0}^{k-1}\EE\left(\delta_{p} X|X_p \right)\right|^2 \leq Ch_n^2 k\sum_{p=0}^{k-1}(1+ |X_p|^2), $$
and so 
\be\label{sommadeidrift}\EE  \left|\sum_{p=0}^{k-1}\EE\left(\delta_{p} X|X_p \right)\right|^2 \leq Ch_n^2 k^2\left(1+ \sup_{p=0,\hdots, N} \EE |X_p|^2\right).
\ee
By Proposition \ref{second_moments_bounded},  \eqref{ypcuadri}, \eqref{sommadeidrift}, \eqref{decomposicion} and our assumption $\rho_n^2= O(h_n)$, we get the existence of $C>0$ such that \eqref{equicontinuity_discrete_times} holds true.\smallskip\\
{\it Step 2: proof of \eqref{equicontinuity_in_d2}: } Let $0 \leq s < t \leq T$ and $k'$, $k$ such that $ s \in [t_{k'}, t_{k'+1}[$ and $ t \in [t_{k}, t_{k+1}[$. Then, by the triangular inequality
\be\label{triangular_inequality}
d_{2}(m^n(t), m^n(s)) \leq d_{2}(m^n(t), m^n(t_k))+d_{2}(m^n(t_k), m^n(t_{k'+1}))+ d_{2}(m^n(t_{k'+1}), m^n(s)). 
\ee
By the dual representation of $d_2^2(\cdot, \cdot)$ (see \cite[Theorem 1.3]{Villani03}), this function is convex in $\P_{2}(\RR^{d})\times\P_{2}(\RR^{d}) $. Thus, relations \eqref{time_interpolation} and \eqref{equicontinuity_discrete_times} imply  that 
 $$d_{2}^2(m^n(t), m^n(t_k))  \leq \left(\frac{t-t_{k}}{h_n}\right)d_{2}^2(m^n(t_{k+1}), m^n(t_k))\leq C^2(t-t_{k}),$$
 from which 
\be\label{ineq_inter_1}
 d_{2}(m^n(t), m^n(t_k)) \leq C (t-t_{k})^{\half}.
\ee
 Analogously, 
\be\label{ineq_inter_2}
d_{2}(m^n(t_{k'+1}), m^n(s)) \leq C (t_{k'+1}-s)^{\half}.
\ee
Relations \eqref{triangular_inequality}, \eqref{ineq_inter_1}, \eqref{ineq_inter_2} and the Cauchy-Schwarz  inequality imply the existence of $C>0$, independent of $n$, such that  
$$
d_{2}(m^n(t),m^n(s) )\leq C|t-s|^{\half}.
$$
Relation \eqref{equicontinuity_in_d2} follows. 
\end{proof}
For notational convenience, for all $\varphi \in C_{0}^{\infty}(\RR^d)$ let us set  \small
\be\label{generator_L}
L_{b,\sigma,\varphi}[\mu](x,t):=\half \sum_{i,j} a_{i,j}[\mu](x,t) \partial_{x_i, x_j}^{2} \varphi(x)+ b[\mu](x,t) \cdot \nabla \varphi(x)  
\hspace{0.2cm} \forall \; (\mu, x,t)\in C([0,T];\P_1(\RR^d)) \times \RR^d \times [0,T].\ee \normalsize
We have now all the elements to prove our main convergence results.  We consider  first the case where, in addition to {\bf(H)},  the coefficients satisfy the following local Lipschitz property:\medskip\\
{\bf(Lip)} For any $\mu \in C([0,T];\P_{1}(\RR^{d}))$  and a compact set $K\subseteq \RR^d$,   there exists a constant $C_{K}>0$ such that
\be\label{local_Lipschitz_inequality}
|b[\mu](y,t)-b[\mu](x,t)| +|\sigma[\mu](y,t)-\sigma[\mu](x,t)| \leq C|y-x|  \hspace{0.2cm} \forall  \;  x, \; y \in  K, \; \;   t\in [0,T].
\ee \smallskip

The case of more general coefficients satisfying only {\bf(H)} will be treated just after. 

\begin{theorem}\label{convergencia_1} Assume ${\bf(H)}$-${\bf(Lip)}$ and that $\rho_n^2=o(h_n)$. Then, every limit point $m \in C([0,T]; \P_1(\RR^d))$ of $m^n$ {\rm(}there exists at least one{\rm)} solves $(FPK)$. In particular,  $(FKP)$ admits at least one solution.
\end{theorem}
\begin{proof} By Proposition \ref{second_moments_bounded}, Proposition \ref{masmammadapppqpqwppa}  and Lemma \ref{compactness_in_P_1}, with $\C= \{m^n \; ; \; n \in \NN\}$, the sequence $m^n$ has at least one limit point $m$. We use the same superscript $n$ to index a   subsequence $m^n$ converging to $m$ in $C([0,T]; \P_1(\RR^d))$ and we need to show that $m$ satisfies \eqref{solution_FP}. Let  $t\in ]0,T]$ and, wihtout loss of generality, consider a sequence $t_{n'}=n'h_n$ such that $t \in ]t_{n'}, t_{n'+1}]$. Then, for every $\varphi \in C_{0}^{\infty}(\RR^{d})$
\be\label{sumatelescopica}
\int_{\RR^{d}} \varphi(x) \dd m^{n}(t_{n'})(x)= \int_{\RR^{d}} \varphi(x) \dd m^{n}(0)(x)+  \sum_{k=0}^{n'-1} \int_{\RR^{d}} \varphi(x) \dd \left[m^{n}(t_{k+1})-m^{n}(t_{k})\right](x).
\ee
For all $k=0, \hdots, n'-1$ we have that 
\be\label{several_computations}\ba{l}
 \int_{\RR^{d}} \varphi(x) \dd m^{n}(t_{k+1})(x)= \sum_{i\in \ZZ^{d}} \varphi(x_{i}) m_{i,k+1}^{n} \\[8pt]
 \hspace{3.5cm} = \sum_{j \in \ZZ^{d}} m_{j,k}^{n}\frac{1}{2r} \sum_{\ell=1}^{r}\sum_{i \in \ZZ^d}\varphi(x_i)  \left[ \beta_{i}( \Phi_{j,k}^{\ell,+}[m^n])+ \beta_{i}( \Phi_{j,k}^{\ell,-}[m^n])\right] \\[8pt]
 \hspace{3.5cm} =  \sum_{j \in \ZZ^{d}} m_{j,k}^{n} \frac{1}{2r} \sum_{\ell=1}^{r} \left[ I[\varphi](\Phi_{j,k}^{\ell,+}[m^n])+  I[\varphi](\Phi_{j,k}^{\ell,-}[m^n])\right]\\[8pt]
 \hspace{3.5cm} = \sum_{j \in \ZZ^{d}} m_{j,k}^{n}  \frac{1}{2r} \sum_{\ell=1}^{r}  \left[ \varphi (\Phi_{j,k}^{\ell,+}[m^n])+  \varphi(\Phi_{j,k}^{\ell,-}[m^n])\right] +  O(\rho_{n}^{2})
 \\[8pt]
 \hspace{3.5cm} =  \sum_{j \in \ZZ^{d}} m_{j,k}^{n}\left[\varphi(x_{j})+ h_n L_{b,\sigma,\varphi}[m^{n}](x_{j},t_k)\right] + O\left( \rho_n^2 +h_n^2\right)  \\[8pt]
 \hspace{3.5cm} =\int_{\RR^{d}} \left[\varphi(x)+ h_nL_{b,\sigma,\varphi}[m^{n}](x,t_k)\right] \dd m^{n}(t_{k})(x) +O\left( \rho_n^2 + h_n^2\right), 
 \ea\ee
where we have used a fourth order Taylor expansion for the terms $\varphi (\Phi_{j,k}^{\ell,+}[m^n])$ and $\varphi(\Phi_{j,k}^{\ell,-}[m^n])$. As a consequence,  \eqref{sumatelescopica} yields  \small
\be\label{escriturasimplificada}
\int_{\RR^{d}} \varphi(x) \dd m^{n}(t_{n'})(x)= \int_{\RR^{d}} \varphi(x) \dd m^{n}(0)(x)+  h_n \sum_{k=0}^{n'-1} \int_{\RR^{d}}L_{b,\sigma,\varphi}[m^{n}](x,t_k) \dd m^{n}(t_{k})(x)+ O\left( \frac{\rho_n^2}{h_n} + h_n\right).
\ee \normalsize
Assumption ${\bf(H)}(i)$ implies the existence of a modulus of continuity  $\bar{\omega}_1$, independent of $k$,  such that 
\be\label{replacement_integral_L_I}
\int_{\RR^{d}}L_{b,\sigma,\varphi}[m^{n}](x,t_k) \dd m^{n}(t_{k})(x)=\int_{\RR^{d}}L_{b,\sigma,\varphi}[m](x,t_k) \dd m^{n}(t_{k})(x)+ \bar{\omega}_1\left( \sup_{t\in [0,T]} d_1(m^n(t),m(t))\right).
\ee
Since $\phi$ has a compact support, condition {\bf(Lip)} implies that  $L_{b,\sigma,\varphi}[m](\cdot,t_k)$ is Lipschitz, uniformly in $k$. Thus, by \eqref{distance_1_difference_lipschitz} and \eqref{equicontinuity_in_d2}, we have
\be\label{replacement_integral_L_II}
\left|\int_{\RR^{d}}L_{b,\sigma,\varphi}[m](x,t_k) \dd \left(m^{n}(s)-m^{n}(t_{k})\right)(x) \right| \leq Cd_{1}(m^{n}(s),m^{n}(t_k))\leq C' \sqrt{h_n} \hspace{0.4cm} \forall \; s \in [t_{k}, t_{k+1}),
\ee
for some positive constants $C$ and $C'$, independent of $n$. This implies that 
$$
\left|h_n \int_{\RR^{d}}L_{b,\sigma,\varphi}[m](x,t_k)\dd m^{n}(t_{k})(x) - \int_{t_{k}}^{t_{k+1}} \int_{\RR^{d}}L_{b,\sigma,\varphi}[m](x,t_k)\dd m^{n}(s)(x)\right| =O\left(h_n^{\frac{3}{2}}\right).
$$
Therefore, by \eqref{escriturasimplificada},  
\be\label{escriturasimplificada_caso_continuo}\ba{ll}
\int_{\RR^{d}} \varphi(x) \dd m^{n}(t_{n'})(x)=& \int_{\RR^{d}} \varphi(x) \dd m^{n}(0)(x)+  \int_{0}^{t_{n'}} \int_{\RR^{d}}\hat{L}_{b,\sigma,\varphi}^n[m](x,s) \dd m^n(s)(x) \dd s\\[8pt]
\; &  + O\left( \frac{\rho_n^2}{h_n} + \sqrt{h_n} +\bar{\omega}_1\left( \sup_{t\in [0,T]} d_1(m^n(t),m(t))\right)\right),\ea
\ee  
where 
$$
\hat{L}_{b,\sigma,\varphi}^n[m](x,s):= L_{b,\sigma,\varphi}[m](x,t_{k}) \hspace{0.4cm} \forall \; x \in \RR^d, \; \; s \in [t_{k}, t_{k+1}).
$$
By {\bf(H)}{\rm(i)}, and the fact that $\phi$ has compact support, we have that $\hat{L}_{b,\sigma,\varphi}^n[m](\cdot, \cdot)$ is uniformly bounded in $n$ and converges uniformly  to $L_{b,\sigma,\varphi}[m](\cdot, \cdot)$ in $\RR^d\times [0,T]$. As a consequence,  for each $s \in [0,T]$, we have that  $\int_{\RR^{d}}\hat{L}_{b,\sigma,\varphi}[m](x,s) \dd m^n(s)(x)$ is uniformly bounded and converges, as $n \to \infty$, to $\int_{\RR^{d}}L_{b,\sigma,\varphi}[m](x,s) \dd m(s)(x)$. Therefore, by Lebesgue's dominated convergence theorem, the second term in the right hand side of \eqref{escriturasimplificada_caso_continuo} converges to 
$$
\int_{0}^{t} \int_{\RR^{d}}L_{b,\sigma,\varphi} [m](x,s) \dd m(s)(x) \dd s. 
$$
%
Finally, passing to the limit in  \eqref{escriturasimplificada_caso_continuo}, we get that \eqref{solution_FP} holds true.
\end{proof}

In the remainder of this section, we consider the case where $b$ and $\sigma$ satisfy only assumption {\bf(H)}. Since in the proof Theorem \ref{convergencia_1} the local Lipchitz assumption {\bf(Lip)} plays an important role,  in the present case we need to regularize the coefficients, which will be done by convolution with a mollifier. Let $\phi \in C^{\infty}(\RR^d)$ have a compact support contained in the closed unit ball $B(0,1):= \{x \in \RR^d \; ; \; |x| \leq 1\}$ and, given a sequence $\eps_n$, with $0< \eps_n\leq 1$, set $\phi_{\eps_n}(x) := \phi(x/{\eps_n})/(\eps_n)^d$ for all $x\in \RR^d$. Let us define 
$$
b_{n}[\mu](x,t):= \phi_{\eps_n} \ast b[\mu](x,t) \hspace{0.3cm} \mbox{and } \; \;  \sigma_{n}[\mu](x,t):= \phi_{\eps_n} \ast \sigma[\mu](x,t),
$$
where the convolution is applied in the space variable $x$ and componentwise for the coordinates of $b$ and $\sigma$. It is easy to check that for each $\mu \in C([0,T]; \P_1(\RR^d))$ and each compact set $K\subseteq \RR^d$, we have that $b_{n}$ and $\sigma_{n}$ satisfy \eqref{local_Lipschitz_inequality} with $C_K= C_{K}'/\eps_n$, where $C_K'$ depends only on $\phi$ and $$\sup \left\{ |b[\mu](x,t)| + |\sigma[\mu](x,t)| \; | \;   x\in K+B(0,1), \; t\in [0,T]\right\}<\infty.$$
We consider the approximation \eqref{scheme_nonlinear_stochastic_case_II} of $(FPK)$ with $\Phi_{i,k}^{\ell,+}[\mu]$ and $\Phi_{i,k}^{\ell,-}[\mu]$ replaced by 
$$\ba{rcl}
\Phi_{i,k}^{n,\ell,+}[\mu]  &:=&  x_{i} + h_n b_n[\mu](x_{i},t_{k})+ \sqrt{r h_n}(\sigma_n)_{\ell}[\mu](x_{i},t_{k}), \\[6pt]
\Phi_{i,k}^{n,\ell,-}[\mu]  &: =&  x_{i} + h_n b_n[\mu](x_{i},t_{k})-\sqrt{r h_n }(\sigma_n)_{\ell}[\mu](x_{i},t_{k}),
\ea
$$
respectively. Namely,  find $m\in \SS^{\rho,h}$ such that
\be\label{scheme_nonlinear_stochastic_case_II_approximated}\ba{l}
m_{i,0} = \bar{m}_{0} (E_{i})  \hspace{0.4cm} \forall \; i\in \ZZ^d,  \\[6pt]
m_{i, k+1} = \frac{1}{2r}\sum\limits_{\ell=1}^{r}\sum\limits_{j \in \ZZ^{d}} \left[\beta_{i} ( \Phi_{j,k}^{n,\ell,+}[m])+\beta_{i}(\Phi_{j,k}^{n,\ell,-}[m])\right] m_{j,k} \hspace{0.4cm} \forall \; i\in \ZZ^d, \; \; k=0,\hdots, N-1. 
\ea \ee 

The coefficients $b_n$ and $\sigma_n$ satisfy {\bf(H)} and the linear growth condition \eqref{linear_growth} holds with a constant $C$ independent of $n$. As a consequence, for each $n\in \NN$, problem \eqref{scheme_nonlinear_stochastic_case_II} admits at least one solution $m^n$ and, denoting likewise the extension of $m^n$ in \eqref{time_interpolation} to an element in $C([0,T]; \P_1(\RR^d))$,  by   \eqref{secmom} and \eqref{equicontinuity_in_d2}, whose proofs can be reproduced without modifications and with constants independent of $n$,  the set $\{m^n \; | \; n\in \NN\}$ is relatively compact in $C([0,T]; \P_1(\RR^d))$. 

We have the following convergence result, assuming only {\bf(H)} and whose proof is almost identical to the previous one. 
\begin{theorem}\label{convergencia_under_H} Assume   ${\bf(H)}$  and that $\rho_n^2=o(h_n)$ and $h_n=o(\eps_n^2)$. Then, every limit point $m \in C([0,T]; \P_1(\RR^d))$ of $m^n$ {\rm(}there exists at least one{\rm)} solves $(FPK)$. In particular,  $(FKP)$ admits at least one solution.
\end{theorem}
\begin{proof} Arguing exactly as in the proof of Theorem \ref{convergencia_1}, and using the same notations, we have the existence of $m\in C([0,T]; \P_1(\RR^d))$ such that, up to some subsequence, $m^n\to m$ in $C([0,T]; \P_1(\RR^d))$. Moreover, for each $n \in \NN$ we have \small
\be\label{escriturasimplificada_1}
\int_{\RR^{d}} \varphi(x) \dd m^{n}(t_{n'})(x)= \int_{\RR^{d}} \varphi(x) \dd m^{n}(0)(x)+  h_n \sum_{k=0}^{n'-1} \int_{\RR^{d}}L_{b_n,\sigma_n,\varphi}[m^{n}](x,t_k) \dd m^{n}(t_{k})(x)+ O\left( \frac{\rho_n^2}{h_n} + h_n\right),
\ee \normalsize
where $L_{b_n,\sigma_n,\varphi}$ is given by \eqref{generator_L}, with $b$ and $\sigma$ replaced by $b_n$ and $\sigma_n$, respectively. Estimate \eqref{replacement_integral_L_I} still holds and \eqref{replacement_integral_L_II} changes to \small
 \be\label{replacement_integral_L_III}
\left|\int_{\RR^{d}}L_{b_n,\sigma_n,\varphi}[m](x,t_k) \dd \left(m^{n}(s)-m^{n}(t_{k})\right)(x) \right| \leq \frac{C}{\eps_n}d_{1}(m^{n}(s),m^{n}(t_k))\leq C' \frac{\sqrt{h_n}}{\eps_n} \hspace{0.4cm} \forall \; s \in [t_{k}, t_{k+1}),
\ee \normalsize
for some constants $C$ and $C'$ independent of $n$. Relation \eqref{escriturasimplificada_1} then gives 
\be\label{escriturasimplificada_caso_continuo_I}\ba{ll}
\int_{\RR^{d}} \varphi(x) \dd m^{n}(t_{n'})(x)=& \int_{\RR^{d}} \varphi(x) \dd m^{n}(0)(x)+  \int_{0}^{t_{n'}} \int_{\RR^{d}}\hat{L}_{b_n,\sigma_n,\varphi}[m](x,s) \dd m^n(s)(x) \dd s\\[8pt]
\; &  + O\left( \frac{\rho_n^2}{h_n} + \frac{\sqrt{h_n}}{\eps_n} +\bar{\omega}_1\left( \sup_{t\in [0,T]} d_1(m^n(t),m(t))\right)\right),\ea
\ee  
where $\hat{L}_{b_n,\sigma_n,\varphi}[m](x,s):= L_{b_n,\sigma_n,\varphi}[m](x,t_{k})$ for all  $x \in \RR^d,$ and  $s\in [t_{k}, t_{k+1})$. By {\bf(H)}{\rm(i)} we have that $\hat{L}_{b_n,\sigma_n,\varphi}[m](\cdot,\cdot) \to L_{b,\sigma,\varphi}[m](\cdot,\cdot)$ uniformly in $\RR^d\times [0,T]$ and, passing to the limit in \eqref{escriturasimplificada_caso_continuo_I}, we can conclude as in the  previous proof.  
\end{proof}

\begin{remark} In particular, Theorem \ref{convergencia_under_H}  yields a Peano type existence result for $(FPK)$.  We point out that more general existence results for the $(FPK)$ equation  are proven in the articles \cite{MR3086740,MR3113428}, by using purely analytical techniques. 
\end{remark}

\begin{remark}\label{remark_approximation_general_coefficients} 
{\rm(i)} In the deterministic case $\sigma \equiv 0$, the proof in \cite[Proposition 3.9]{MR3148086}  shows that \eqref{equicontinuity_in_d2} can be replaced by 
$$\sup_{n\in \NN} d_{1}(m^n(t),m^n(s) )\leq C|t-s| \hspace{0.3cm} \forall \;  t, s \in [0,T],$$
and, hence, the estimate \eqref{replacement_integral_L_III} can be improved to \small
$$\left|\int_{\RR^{d}}L_{b_n,\sigma_n,\varphi}[m](x,t_k) \dd \left(m^{n}(s)-m^{n}(t_{k})\right)(x) \right| \leq \frac{C}{\eps_n}d_{1}(m^{n}(s),m^{n}(t_k))\leq C' \frac{h_n}{\eps_n} \hspace{0.4cm} \forall \; s \in [t_{k}, t_{k+1}),$$
\normalsize
for some constants $C$ and $C'$ independent of $n$. As a consequence, the result in Theorem \ref{convergencia_under_H} holds true under the weaker assumption $h_n=o(\eps_n)$.\smallskip\\
{\rm(ii)} The approximation of the coefficients can also be useful in order to approximate the $(FPK)$ equation with coefficients $b$ and $\sigma$ defined almost everywhere w.r.t. the Lebesgue measure. In this case, in order to give a meaning to a solution $m$ of \eqref{solution_FP} one can require that  $m(t)$  should be absolutely continuous w.r.t. the Lebesgue measure for almost every $t\in [0,T]$. One can then consider coefficients $b^n$ and $\sigma^n$ which regularize $b$ and $\sigma$, but in general we can only expect $L^1$ convergence $L_{b^n,\sigma^n,\varphi}$ to $L_{b,\sigma,\varphi}$. In this case, the scheme \eqref{scheme_nonlinear_stochastic_case_I} should be modified in order to discretize the density of $m$ and a stronger compactness result, for example in $L^{\infty}$ endowed with the weak$^*$ topology,  should be proved for the constructed approximation $m^n$. As we will discuss in {\rm Remark \ref{uniqueness_mfg}(ii)}, this is exactly the situation in degenerate MFGs {\rm(}see {\rm\cite{MR3148086,CS15})}. 

\end{remark}
%
\section{Applications and Numerical simulations}\label{aplications}
We describe several applications where our scheme can be efficiently used to approximate the solution of the FPK equation. We consider first  two standard linear models. The first one consists  in a FPK equation where the underlying two-dimensional dynamics models a damped noisy harmonic oscillator. In this case, there is an explicit exact solution, which is helpful in order to test the scheme and compute the numerical errors. In the second linear model we consider a first order FPK equation, where the underlying dynamics describes a predator-prey model under the effect of  a periodic force that models seasonality. In this test we propose a simple modification of the scheme which allows us to simulate the long time behavior of the dynamics by considering  very  large time steps.

Next, we apply our scheme to solve  two non-linear models with $\sigma[m](x,t)\equiv \sigma I_{d}$  for some $\sigma\neq 0$ (where $I_d$ is the $d\times d$ identity matrix), but where $b[m](x,t)$ does not admit  an explicit expression and has to be approximated.  The approximation technique is similar to the one presented at the end of the previous sections, where the coefficients supposed to satisfy {\bf(H)} only.
In the first model we consider an example of the so-called MFG system with non-local interactions (see \cite{LasryLions07}).  In this case, the drift $b[m](x,t)$ is related to the value function of an optimal control problem starting at $x$ at time $t$, having  running and terminal costs depending on $\{m(s) \; ; s \in ]0,T[\}$ and $m(T)$, respectively. Therefore, as explained in Remark \ref{explication_explicito_implicito}, the proposed scheme is implicit. Our approximation is similar to the one in \cite{MR3148086,CS13,CS15} dealing with degenerate  MFG systems and where the authors prove the convergence when the state dimension  $d$ is equal to one. In our present non-degenerate setting, the theory developed in Section \ref{convergence_section} allows us to prove the convergence of the scheme in general space dimensions. In the second non-linear model, we consider a FPK equation where the velocity field $b[m](x,t)$  depends on the value function of an optimal control starting at $x$ at time $t$ with   running and terminal costs depending only on the value $m(t)$. This model, which seems to be new, is inspired by the Hughes model \cite{hughes2000flow} and could be used to model crowd motion in some ``panic'' situations. We prove that the related FPK equation admits at least one solution and we also provide a convergence result for the associated scheme.
%

\subsection{Linear case:  damped noisy harmonic oscillator}
We consider the numerical resolution of a FPK equation modeling a harmonic oscillator with damping coefficient $\gamma>0$ and noise coefficient  $\sigma >0$. The dynamics  is  described by the following two dimensional SDE in an interval $]0,T[$
\be\label{DHO}
\ba{rcl}
\dd X_1(t) &=& X_2(t)\dd t \\[4pt]
\dd X_2(t) &=&\left[ -X_1(t)-\gamma X_2(t)\right]\dd t +\sqrt{2\sigma} \dd W(t),\\[4pt]
(X_1(0), X_2(0))&=&(\bar{X}_1(0), \bar{X}_2(0)) \hspace{0.3cm} \mbox{with}\hspace{0.2cm} \mbox{Law}((\bar{X}_1(0), \bar{X}_2(0)))=\bar{m}_0 \in \P_2(\RR^2),
\ea
\ee
and $(\bar{X}_1(0), \bar{X}_2(0))$ independent of the one-dimensional Brownian motion $W$. 
The associated (degenerate) FPK equation is 
\be \label{FPKosc} 
\partial_t m -\sigma \partial_{x_2,x_2}m+\partial_{x_1}(x_2 m) -\partial_{x_2}((x_1+\gamma x_2)m)=0 \hspace{0.3cm} \mbox{in } \RR^2 \times ]0,T[,\hspace{0.3cm} m(0)=\bar{m}_0. 
\ee
Supposing that $\bar{m}_0:=\delta_{x_0}$ ($x_0\in \RR^2$), it is shown in \cite{ZMV99}  that the  solution  $m$ to \eqref{FPKosc} has a density, which has the following explicit expression \be\label{solex}
m(x,t)=\frac{\nu(x,t)}{ \int_{\RR^d} \nu(y,t) \dd y}, \; \; \; \mbox{where } \; \; \nu(x,t):= \frac{e^{\gamma t-s_{x_0}(x,t)/2\Delta(t)}}{2\pi \sqrt{\Delta(t)}},
\ee
with 
$$\ba{rl} s_{x_0}(x,t):=& a(t)(\psi(x,t)-\psi(x_0,0))^2+2H(t)\left[\psi(x,t)-\psi(x_0,0)\right]\left[\eta(x,t)-\eta(x_0,0)\right]\\[6pt]
\; & +b(t)(\eta(x,t)-\eta(x_0,0))^2,\\[6pt]
\Delta(t):=&a(t)b(t)-H(t)^2,
\ea$$
and 
$$
\ba{rl}
\psi(x,t):=&(x_1\mu_1-x_2)e^{-\mu_2 t},\quad \eta(x,t):=(x_1\mu_2-x_2)e^{-\mu_1 t},\\[6pt]
H(t):=&-\frac{2\sigma}{\mu_1+\mu_2}(1-e^{-(\mu_1+\mu_2)t}),\\[8pt]
a(t):=&\frac{\sigma}{\mu_1}(1-e^{-2\mu_1 t}),\quad
b(t):=\frac{\sigma}{\mu_2}(1-e^{-2\mu_2t}),\\[6pt]
\mu_1:=&-\frac{\gamma}{2}+(\frac{\gamma^2}{4}-1)^{\frac{1}{2}},\quad\mu_2:=-\frac{\gamma}{2}-(\frac{\gamma^2}{4}-1)^{\frac{1}{2}}.\\
\ea
$$ \vspace{0.2cm}

We apply our  scheme to approximate the solution of  \eqref{FPKosc} in the time interval  $[0,T]:= [0,2]$  with $\gamma=2.1$,  $\sigma=0.8$ and $\bar{m}_0:=\delta_{x_0}$ with $x_0:=(1,1)$. Since most of the support of the exact solution $m$ is contained in $\OO:=(-4,4)^2$, we consider the solution of our scheme restricted to this domain (which implies that the total mass is not conserved) in order to obtain an implementable method. An alternative would be to impose Neumann boundary conditions (see the next example) in order to maintain the total mass constant. However, in that case we loose the explicit expression \eqref{solex} for the exact solution.

%
Given $\rho$, $h=T/N>0$ ($N\in \NN$), and the  weights  $m_{i,k}$ ($i\in \ZZ^2$, $k=0, \hdots,N$),  defined  recursively by \eqref{scheme_linear_stochastic_case}, we set $\mathbf{m}_{\rho,h}(x,t):=m_{i,k}/\rho^2$ if $(x,t)\in E_i\times  [t_k,t_{k+1})$, which, for fixed $t$,  defines a density which is uniform on $E_i$. Let us set
\be \label{err}
\mathcal{E}_{\rho,h}:=\left[\frac{1}{K^2}\sum_i(\mathbf{m}_{\rho,h}(x_i,T)-m(x_i,T))^2\right]^{\frac{1}{2}},\ee
where $K$ is the total number of grid nodes. The value  $\mathcal{E}_{\rho,h}$ measures a discrete $L^2$ error  between the density of $m$ and its approximation.  Note that the convergence theory presented in Section \ref{convergence_section} does not imply that $\mathcal{E}_{\rho,h}$ should tend to $0$ as $\rho$ and $h$ tend to zero. Nevertheless, we observe this behavior numerically. Indeed, for $\rho=0.1$, $0.05$, $0.025$ we set $h=\rho/2$ and compute $\mathcal{E}_{\rho,h}$ for the corresponding numerical approximations. In the first two columns of Table \ref {tb:test1} we show the selected  parameters. In the third and fourth columns we show the associated error  $\mathcal{E}_{\rho,h}$  and the  convergence rate, respectively.  
 In Figure \ref{TestDO},  we  display on the left the contour level set of $\mathbf{m}_{\rho,h}(\cdot,t)$ at the level $0.2$, defined as $\Gamma_t:=\{x \in \OO \; ; \; \mathbf{m}_{\rho,h}(x,t)=0.2\}$, and computed at  times $t=0.2$, $0.5$, $1$, $2$   with $\rho=0.025$. To the right in the same figure,  we provide a  3D view of the numerical solution computed  at the final time $T=2$ with $\rho=0.025$.
Even in this simple linear setting, this test shows two main advantages of our scheme. Compared to  explicit finite difference schemes, the  discretization we propose is  stable, explicit and, at the same time, allows large time steps. Moreover, it can handle initial data with very weak regularity  (a Dirac mass in this particular case).
 
 \begin{table}[ht!]
\begin{center}
\caption{Damped Oscillator: $\mathcal{E}_{\rho,h}$ errors and convegence rate }
\label{tb:test1}
\begin{tabular}{|c|c|c|c|}\hline
$\rho$  &$h$  & $\mathcal{E}_{\rho,h}$  & convergence rate\\ \hline
$0.1$&$0.05$&$1.02 \cdot 10^{-2}$& --  \\ 
$0.05$&$0.025$&$5.37 \cdot 10^{-3}$& $0.93 $ \\
$0.025$&$0.0125$&$2.45 \cdot 10^{-3}$&$ 1.12  $\\  \hline
\end{tabular}
\end{center}
\end{table}
\begin{figure}[htp]
\begin{center}
  \includegraphics[width=2.5in]{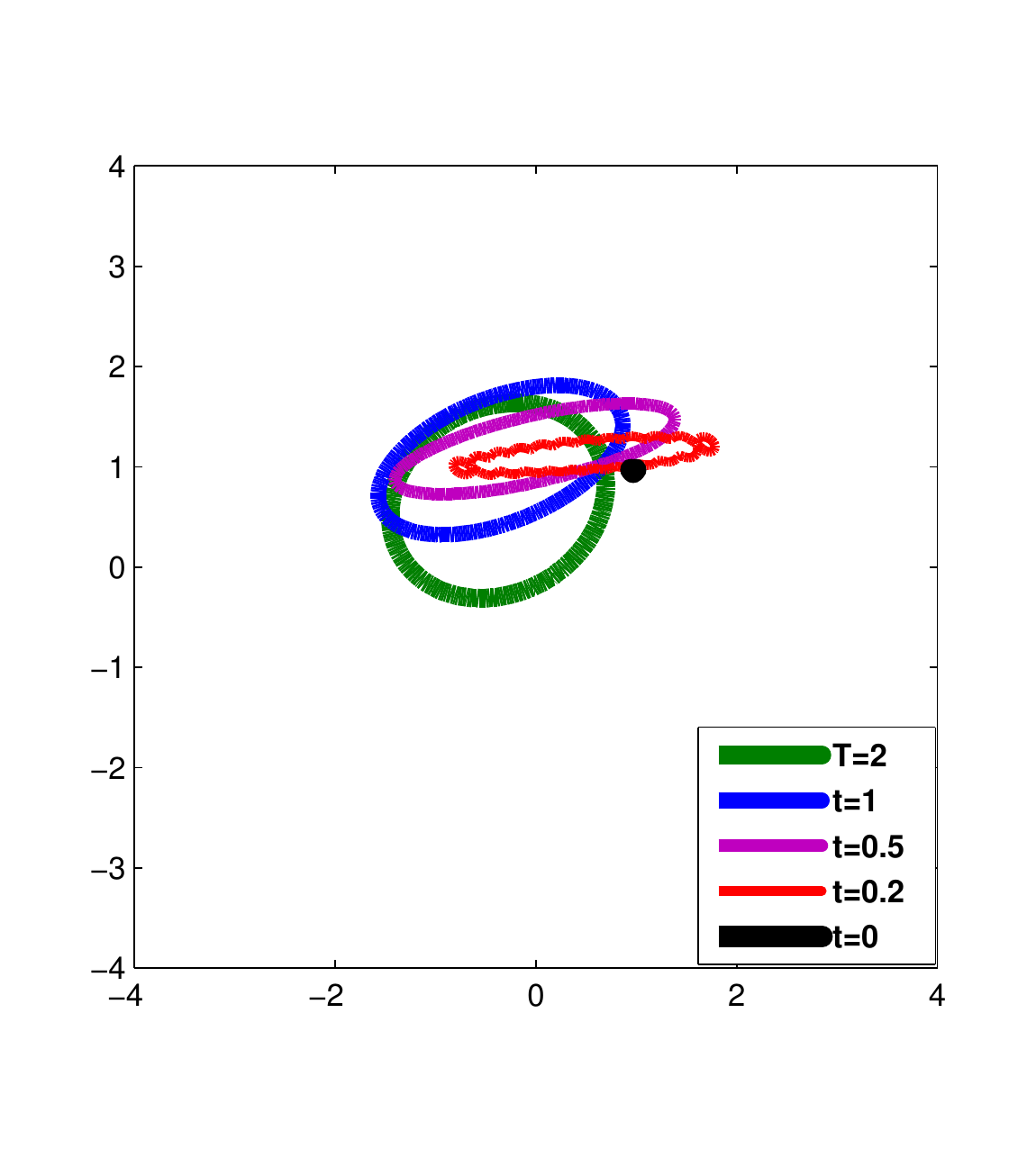}
 \includegraphics[width=3.5in]{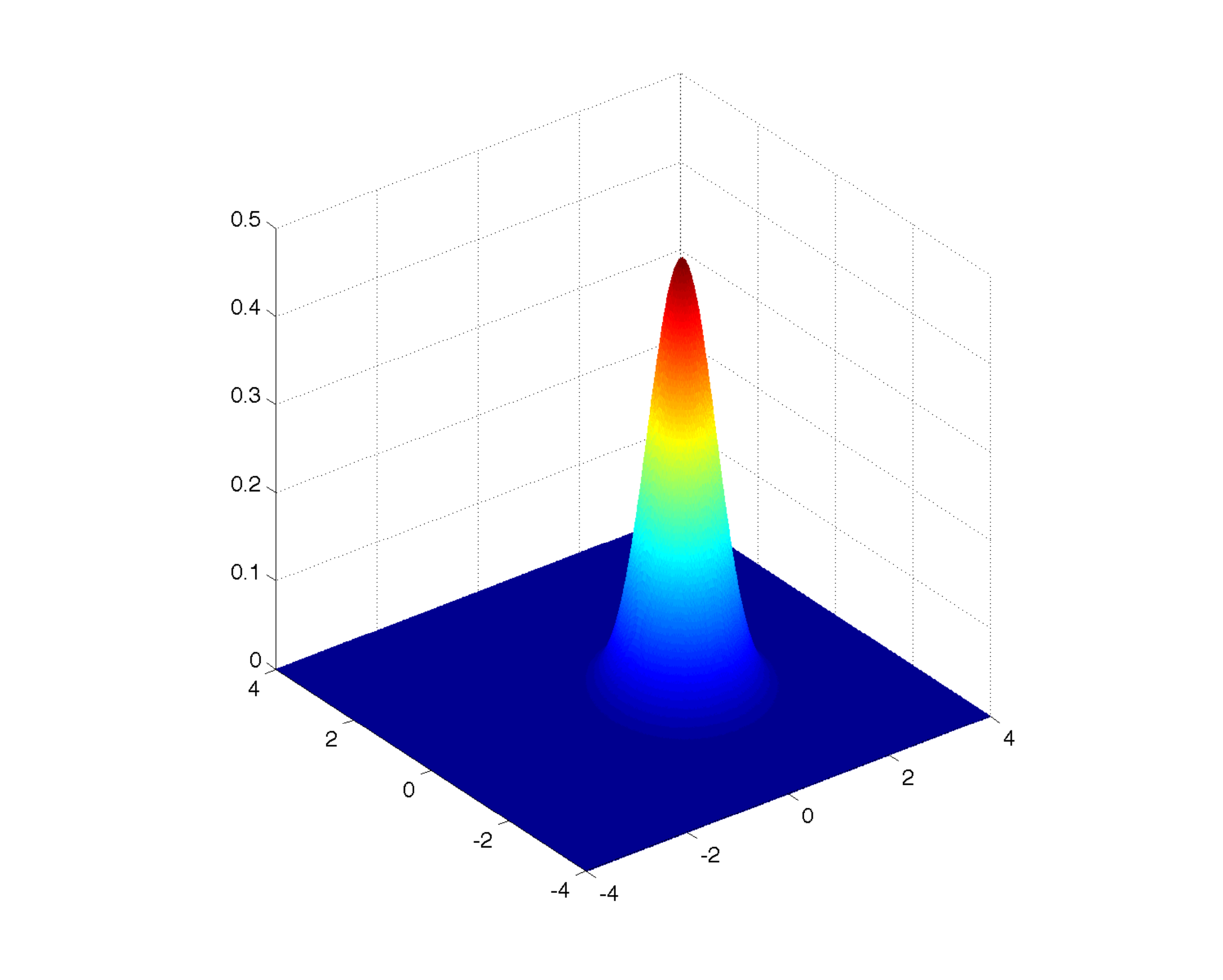}
\caption{{ Damped oscillator: On the left we display the contour level sets for $\mathbf{m}_{\rho,h}(x,t) =0.2$ at times $t=0.2$, $0.5$, $1$ and $2$.  The black point corresponds to $(1,1)$, which is the point where the initial mass is concentrated. On the right, we display    a  3D view  of the  numerical solution at time $T=2$  computed with $\rho=0.025$.}}
\label{TestDO}
\end{center}
\end{figure}

\subsection{Linear and deterministic case: Lotka-Volterra model with seasonality}
We consider now a Lotka-Volterra type system that models the time evolution of a two-species predator-prey system under the effect of seasonality (see \cite{KS99}).   The number of predators and preys, as functions of time, are denoted by $U$ and $V$, respectively.  The dynamics of $(U,V)$ in the time interval $[0,+\infty[$ is  described by  (omitting the initial conditions)
\be\label{lotka_volterra_UV}
\ba{rcl} 
\dd U(t) &=&\left[ -U(t) +U(t) V(t)\right]\dd t\\[4pt]
\dd V(t)&=& \left[(1+\lambda \sin(t))V(t) -U(t)V(t)-\gamma V(t)^2\right]\dd t,			
\ea 
\ee
where $\lambda \geq 0$ and $\gamma>0$.  The predators have death and growth  rates  equal to 1. The  preys have  death rates equal to 1, due to the presence of predators,   but they  are also affected by self-limitations effects  (due, for instance, to resource limitation)  which are modeled   by the term $\gamma V(t)^2$. 
The growth rate of the preys has periodic variations  $t\mapsto 1+\lambda\sin(t)$  to model seasonality. 
If $\lambda =0$, system \eqref{lotka_volterra_UV} has a unique non trivial positive equilibrium, while in the seasonal case $\lambda>0$ the equilibrium is shown to be a periodic orbit around the origin. We refer the reader to  \cite{KS99} for analytical details on this model. The system can be simplified by the logarithmic transformation $X_1=\ln U, X_2=\ln V$
into
\be\label{LVb}
\ba{rcl} 
\dd X_1(t) &=&\left[ -1 +e^{X_2(t)}\right]\dd t\\[4pt]
\dd X_2(t) &=& \left[1+\lambda\sin(t)-e^{X_1(t)}-\gamma e^{X_2(t)}\right]\dd t.		
\ea 
\ee
Note that the coefficients defining \eqref{LVb} do not satisfy the growth assumption {\bf(H){\rm(ii)}}. Despite this fact, we will show next that the scheme we propose approximates correctly the associated FPK equation.
\subsubsection{Numerical simulation}
We  numerically solve the associated first order linear FPK equation (or  continuity equation) with $\lambda=0.05$ and $\gamma=0.05$ on the bounded domain $\OO \times [0,T]:=[-1.5,1,5]^2\times [0,150]$
and with an absolutely continuous initial condition with density given by
$$ \bar{m}_0(x)=\frac{\nu(x)}{\int_{\OO}\nu(y)\dd y}  \mathbb{I}_{\OO}(x) \; \, \;  {\rm with} \; \;   \nu(x_1,x_2):= e^{\frac{-(x_1-0.4)^2-(x_2-0.4)^2}{0.05}},$$
and $\mathbb{I}_{\OO}(x)=1$, if $x\in \OO$, and $\mathbb{I}_{\OO}(x)=0$, otherwise. Since we consider a bounded space domain, we complement the FPK equation with an homogeneous Neumann boundary condition which, in terms of the underlying characteristics, means that  trajectories are reflected once they touch the boundary. As a consequence, the total mass is preserved during the evolution.  Accordingly,  at the level of  the fully-discrete scheme we reflect the discrete characteristics. This modification of the scheme is detailed discussed in \cite{carlini2016PREPRINT}, in the context of Hughes model for pedestrian flow (see  \cite{hughes2000flow}). Let us point out, that a theoretical study of the convergence  of the resulting scheme has not yet been established and remains as an interesting subject of future research. 

Since the time horizon $T=150$ is long, in order to allow large time steps and maintain the accuracy of the numerical method we  modify  our scheme in the following way. We define  a second  time step $\delta>0$, such that $h=P\delta$, with $P\in  \NN$. This new time step is used to compute the discrete flow \eqref{discrete_trajectories_deterministic_case}, at each node $x_i$ on each  time interval $[t_k,t_{k+1}]$ of size $h$, in the following way:
$$\Phi_{i,k}:=z_{k}^P(x_i)
$$ where $z_{k}^P(x_i)$  is the discrete trajectory computed after $P$ iterations of the Euler scheme with time step $\delta$:   $z_k^0=x_i$ and
$z_{k}^{p+1}(x_i)=z_{k}^p(x_i)+\delta b(z_{k}^p(x_i),t_k+p\delta)$   ($p=0,\hdots,P-1$), with 
\be\label{drift_lotka_volterra}
b(x,t):= \left( -1 +e^{x_2}, 1+\lambda\sin(t)-e^{x_1}-\gamma e^{x_2}\right).
\ee  
Defining $\mathbf{m}_{\rho,h}$ as in the previous example, in  Figure \ref{Test2} we show the time averaged  density computed on the time interval $I_T=[100,150]$ by the formula $\overline{ \mathbf{m}}^{\rho,h}(x)=\frac{1}{|I_T|}\sum_{t_k\in I_T} \mathbf{m}^{\rho,h}(x,t_k)h$   with $\rho=0.015$, $h=8\rho$ and $P=16$. 

Let us point out that in \cite{NADIGA08} the authors implement a path integration method for  a FPK equation associated to a stochastic Lotka-Volterra system whose drift $b$ is 
given by \eqref{drift_lotka_volterra}.  Due to  the absence of the diffusion term in    system \eqref{LVb},   we observe that the   approximated time average density in Figure \ref{Test2}  is more concentrated  than the one displayed in \cite{NADIGA08}.  On the other hand, the shapes of the  periodic orbits are very similar in both cases.
\begin{figure}[ht]
\begin{center}
\includegraphics[width=6cm]{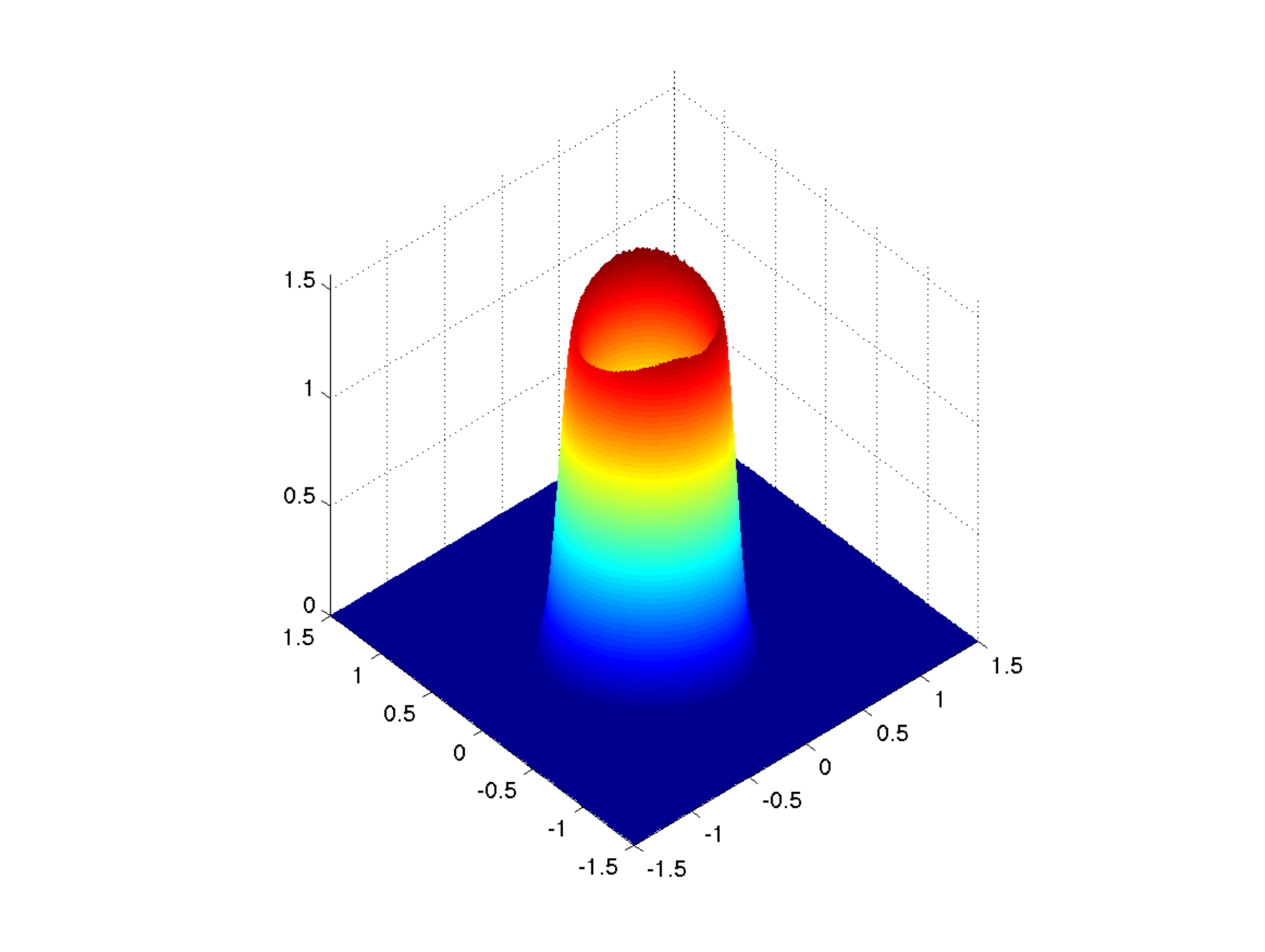}\includegraphics[width=6cm]{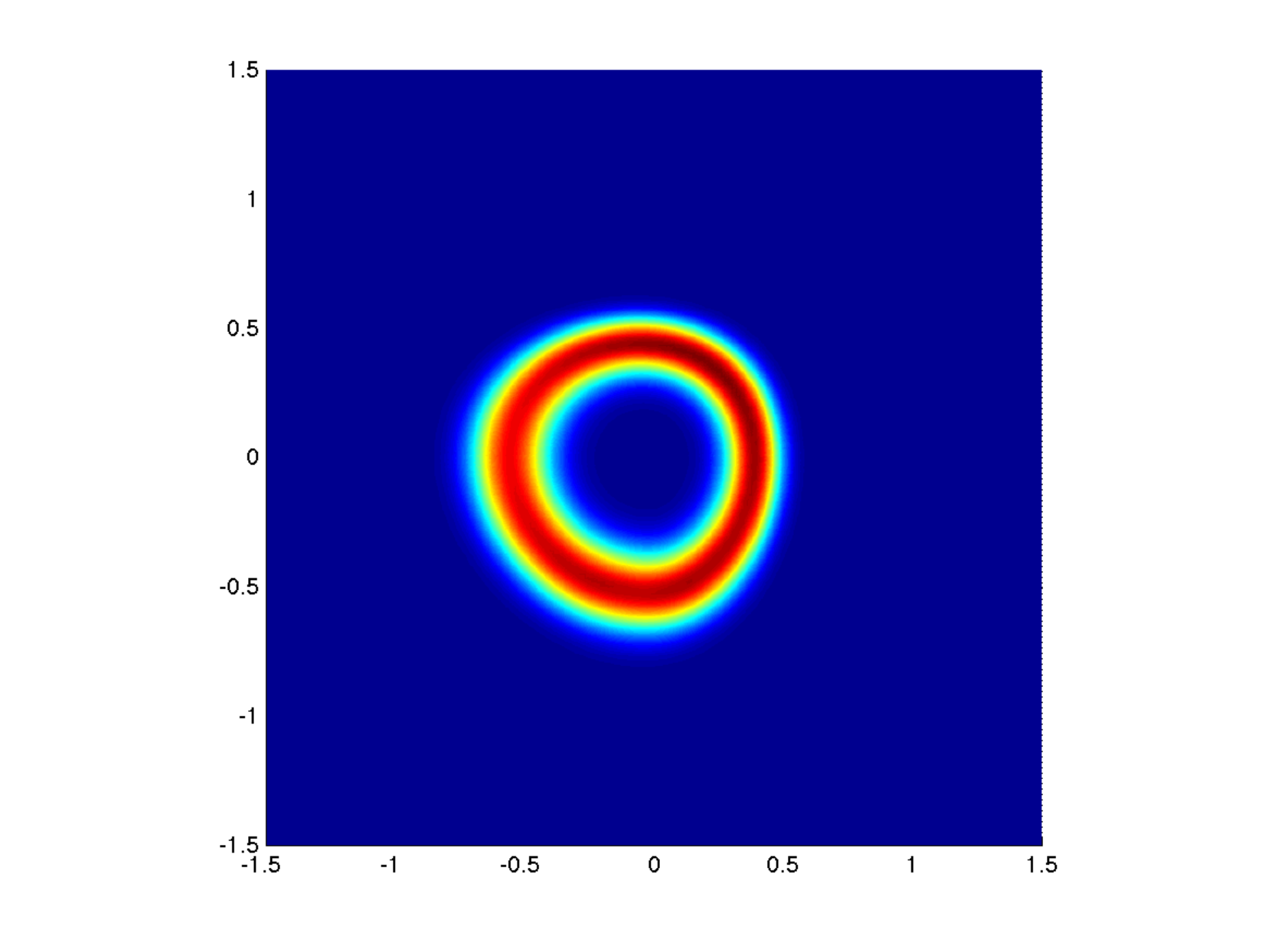}
{\caption{Time averaged solution $\overline{ \mathbf{m}}_{\rho,h}(x)$ computed  with $\rho=0.015,h=8\rho,P=16$, using parameters  $\lambda=0.05$ and $\gamma=0.05$.
\label{Test2}}}
\end{center}
\end{figure}


\subsection{Mean Field Games as a non-linear implicit model}\label{subsection_mfg} We consider here the MFG system 
\small
\be\ba{rcl}\label{MFGimplicit} 
-\partial_{t} v   -  \frac{\sigma^2}{2} \Delta v    + \frac{1}{2}|\nabla v|^2  &=& F(x, m(t)) \;  \;    \hbox{in } \RR^{d}\times (0,T), \\[6pt]
\partial_{t} m  -\frac{\sigma^2}{2}\Delta m   -\mbox{div} \big( \nabla v m \big) &=&0 \;  \; \; \hbox{in } \RR^{d}\times (0,T), \\[6pt]
v(x,T)= G(x, m(t)) \; \;   \mbox{for } x \in \RR^{d}, &\;&  \, \; \; m(0)=\bar{m}_0(\cdot) \in \P_{2}(\RR^{d}),
\ea
\ee \normalsize
where $\sigma \neq 0$ and  $F$, $G: \RR^d \times \P_{1}(\RR^d) \to \RR$ are continuous, twice differentiable w.r.t. the space variable,  and  satisfy that there exists a constant $c>0$ such that for $\psi= F, G$
\be\label{assumptions_F_G}
\sup_{ x\in \RR^d, \mu \in \P_{1}(\RR^d) }  \left(| \psi(x,\mu)| + |\nabla_{x} \psi(x,\mu)| +|\nabla_{xx}^2 \psi(x,\mu)|\right)  \leq c.
\ee

System \eqref{MFGimplicit} is a particular instance of a generic class of models introduced by Lasry and Lions in \cite{LasryLions06i,LasryLions06ii,LasryLions07} that characterize Nash equilibria of stochastic differential games with an infinite number of players. In order to explain the intuition behind \eqref{MFGimplicit}, for $m \in C([0,T];\P_1(\RR^d))$   consider  the HJB equation 
\be\ba{rcl}\label{HJBmu} 
-\partial_{t} v   -  \frac{\sigma^2}{2} \Delta v    + \frac{1}{2}|\nabla v|^2  &=& F(x, m(t)) \;  \;    \hbox{in } \RR^{d}\times (0,T), \\[6pt]
v(x,T)&=& G(x, m(T))\; \;   \mbox{for } x \in \RR^{d}.
\ea\ee
Standard results in stochastic control (see e.g. \cite{MR2179357}) imply that the  unique solution $v[m]$ of \eqref{HJBmu} can be represented as 
\be\label{stochastic_optimal_control_problem_mfg}
v[m](x,t):= \inf_{\alpha} \; \EE\left( \int_{t}^{T}\left[ \half |\alpha(s)|^2 + F(X^{x,t,\alpha}(s),m(s)) \right] \dd s  + G(X^{x,t,\alpha}(T),m(T)) \right),
\ee
where the expectation $\EE$ is taken in a complete probability space $(\Omega,\F, \PP)$ on which an $r$-dimensional Brownian motion $W$ is defined, the $\RR^{d}$-valued processes $\alpha$ are adapted to the  natural filtration generated by $W$, completed with the $\PP$-null sets, and they satisfy $\EE\left(\int_{0}^{T} |\alpha(t)|^2 \dd t \right)<\infty$, and $X^{x,t,\alpha}$ is defined as the solution of
\be\label{controlled_equation_mfg}
\dd X(s)= \alpha(s) \dd s + \sigma \dd W(s) \, \; s\in (t,T),  \hspace{0.8cm} X(t)=x.
\ee
The optimization problem in \eqref{stochastic_optimal_control_problem_mfg} can be interpreted in terms of a generic small agent whose state is $x$ at time $t$ and optimizes a cost depending on the future distribution of the agents $\{m(s) \; ; \; s\in ]t,T]\}$. 
The solution $v[m]$ of \eqref{HJBmu} is classical (see e.g. \cite{Cardialaguet10} where the proof is based upon the Hopf-Cole transformation) and so, by a formal  verification argument (see e.g. \cite{MR2179357}), the optimal trajectory for $v[m](x,t)$ in \eqref{stochastic_optimal_control_problem_mfg} is given by the solution $X^{x,t}$ of 
\be\label{mckean_vlasov_mfg}
\dd X(s)= -\nabla_{x} v[m]\left(X(s),s \right)\dd s + \sigma \dd W(s)\, \; s\in (t,T),  \hspace{0.8cm}  X(t)=x, 
\ee
and the optimal control $\alpha$ is given in feedback form $\alpha(x,t)=-\nabla_{x} v[m]\left(x,t\right)$. Thus, if all the players, distributed as $m_0$ at time $0$, act optimally according  to this feedback law, then the evolution of $m_0$ will be described by the FPK equation 
$$
\partial_{t} \mu  -\frac{\sigma^2}{2}\Delta \mu   -\mbox{div} \big( \nabla v[m] \mu \big)  = 0 \; \; \; \mbox{in } \; \RR^d\times (0,T) , \hspace{0.3cm} \mu(0)=m_0,
$$
and the equilibrium condition reads $m=\mu$, i.e. 
\be\label{fpk_mfg_implicit}
\partial_{t} m  -\frac{\sigma^2}{2}\Delta m   -\mbox{div} \big( \nabla v[m] m \big)  = 0 \; \; \; \mbox{in } \; \RR^d\times (0,T), \hspace{0.3cm} m(0)=m_0.
\ee
The equilibrium equation \eqref{fpk_mfg_implicit} is a particular instance of $(FKP)$ with $r=d$, $\sigma_{ij} =\sigma$ if $i=j$ and $0$ otherwise, and 
\be\label{definition_b_mfg}
b[\mu](x,t):= - \nabla v[\mu](x,t) \hspace{0.6cm} \forall \; (\mu,x,t) \in C([0,T]; \P_1(\RR^d))\times \RR^d \times [0,T], 
\ee
which depends on $\mu$ non-locally in time through $\{\mu(s) \; ; \; s\in (t,T]\}$ by \eqref{stochastic_optimal_control_problem_mfg} (with $m$ replaced by $\mu$). 
 
Let us now recall some properties of $v$ that allow to check assumption {\bf(H)} for $b$. Note that \eqref{stochastic_optimal_control_problem_mfg}, assumption \eqref{assumptions_F_G} and standard estimates for the solutions of the controlled SDE \eqref{controlled_equation_mfg} imply that $v$ is bounded and continuous. Moreover, $v$ is 
 uniformly semiconcave w.r.t. the space variable (see e.g. \cite{CannSinesbook} and \cite[Chapter 4]{MR2179357}), i.e. there exists $c>0$, independent of $t\in [0,T]$ and $\mu \in C([0,T];\P_1(\RR^d))$, such that  for all $x\in \RR^d$, $\mu \in \P_1(\RR^d)$ and $t\in [0,T]$,
\begin{equation}\label{semiconcavity_1}
v[\mu](x+h,t) -2v[\mu](x,t) +  v[\mu](x-h,t) \leq c|h|^2 \hspace{0.5cm} \forall \; h \in \RR^d,
\end{equation}
or equivalently, since $v[\mu](\cdot, t)$ is differentiable, there exists a constant $c>0$, independent of $t\in [0,T]$ and $\mu \in \P_1(\RR^d)$, such that
\be\label{semoconcavity_2_with_derivative}
v[\mu](x+h,t) \leq v[\mu](x,t) + \nabla_{x} v[\mu](x,t)\cdot h+ c |h|^2 \hspace{0.4cm} \forall \; h \in \RR^d, \; t\in [0,T].
\ee
In addition, the uniform Lipschitz property for $F(\cdot, \mu)$ and for  $G(\cdot, \mu)$  and formulation \eqref{stochastic_optimal_control_problem_mfg} imply, using again the stability results for the solutions of \eqref{controlled_equation_mfg} in terms of the initial condition, that 
\be\label{uniform_Lipschitz_property_v}\sup_{ t\in [0,T], \; \mu \in C([0,T];\P_{1}(\RR^d)) }   \|\nabla_{x}v[\mu](\cdot,t)\|_{\infty}  < \infty.\ee
As a consequence,   the continuity of $v$ yields that for any $(\mu_n,x_n,t_n) \to (\mu,x,t)$ we have that any limit point $p$ of $\nabla_{x} v[\mu_n](x_n,t_n)$ (there exists at least one by \eqref{uniform_Lipschitz_property_v})  must satisfy 
$$
v[\mu](x+h,t) \leq v[\mu](x,t) + p\cdot h+ c |h|^2 \hspace{0.4cm} \forall \; h \in \RR^d, \; t\in [0,T],
$$
and so $p=\nabla_{x} v[\mu](x,t)$ by \cite[Proposition 3.3.1 and Proposition 3.1.5(c)]{CannSinesbook}. Therefore,   $b$, defined in \eqref{definition_b_mfg}, is continuous. Since  \eqref{uniform_Lipschitz_property_v} implies that $b$ is bounded, we have that $b$ and $\sigma$ satisfy ${\bf(H)}$. Moreover, by \eqref{HJBmu} and the fact that $\nabla v[\mu]$ is bounded (independently of $\mu$), standard results for parabolic equations imply that $b$ and $\sigma$ also satisfy {\bf(Lip)}. 

Consequently, the results of Sections \ref{fully_discrete_scheme_section} and \ref{convergence_section} are applicable to \eqref{mckean_vlasov_mfg}. However, from the numerical point of view, we cannot implement the fully-discrete scheme directly with $b$, because we do not have an explicit expression for this vector field, which depends on the value function $v$. To overcome this difficulty, we argue as at the end of Section \ref{convergence_section}, where we approximate $b$ and $\sigma$ satisfying {\bf(H)} by coefficients which are locally Lipschitz, and approximate $b$ by a sequence of computable vector fields. We consider a Semi-Lagrangian scheme for the solution of \eqref{HJBmu} with $m$ replaced by $\mu$. Given $\rho>0$, $h=T/N>0$, with $N\in \mathbb{N}$,  and $\mu \in C([0,T];\P_1(\RR^d))$ we first define $v^{\rho,h}[\mu]$ in $\G_{\rho} \times \{0, \hdots, N\}$ recursively as  \small
\be\label{semi_lagrangian_scheme_v_mu}
\ba{rcl} 
v^{\rho,h}_{i,k}&=& \inf_{\alpha \in \RR^d} \left\{ \frac{h}{2}|\alpha|^{2} + \frac{1}{2d}\sum_{\ell=1}^{d}\left(I[v^{\rho,h}_{\cdot,k+1}](x_{i} + h \alpha +\sigma \sqrt{hd}e_{\ell})+I[v^{\rho,h}_{\cdot,k+1}](x_{i} + h \alpha-\sigma \sqrt{hd}e_{\ell}) \right)\right\}\\[6pt]
\; & \; & \hspace{1.2cm}+ hF(x_i,\mu(t_k)) \hspace{0.4cm} \forall \; k=0, \hdots, N-1,\\[6pt]
v^{\rho,h}_{i,N}&=& G(x_i,\mu(T)),
\ea
\ee \normalsize
where $\{e_\ell \; ; \; \ell=1, \hdots, d\}$ is the canonical basis of $\RR^d$, and we have omitted the $\mu$ dependence of $v^{\rho,h}$. We then define  $v^{\rho,h}: C([0,T]; \P_{1}(\RR^d)) \times \RR^d \times [0,T] \to \RR$ by  
$$
v^{\rho,h}[\mu](x,t)=  I[v^{\rho,h}_{\cdot,k}[\mu]](x,t_k) \hspace{0.5cm} \mbox{if } \; t\in [t_{k}, t_{k+1}[.
$$
In order to get a function differentiable w.r.t. the space variable, given $\eps>0$ and $\phi \in C^{\infty}(\RR^d)$, non-negative and such that $\int_{\RR^{d}} \phi(x)\dd x=1$, let us set $\phi_\eps(x):= \frac{1}{\eps^d} \phi(x/\eps)$. We define $v^{\rho,h,\eps}: C([0,T]; \P_{1}(\RR^d)) \times \RR^d \times [0,T] \to \RR$ by 
$$
v^{\rho,h,\eps}[\mu](\cdot, t):= \phi_{\eps} \ast v^{\rho,h}[\mu](\cdot,t) \hspace{0.3cm} \forall \; t\in [0,T].
$$
In \cite[Lemma 3.2 (i)]{CS15} it is shown that $v^{\rho,h,\eps}[\mu](\cdot, t)$ is Lipschitz, uniformly in $(\rho,h,\eps,\mu,t)$  which shows the bound \eqref{uniform_Lipschitz_property_v} for $v^{\rho,h,\eps}$. Using that $v^{\rho,h}$ satisfies a discrete semiconcavity property (see  \cite[Lemma 3.1 (ii)]{CS15}), by \cite[Lemma 4.3 and Remark 4.4]{MR3180719}   there exists a constant $c>0$,  independent of $(\rho,h,\eps,\mu,t)$, such that  $v^{\rho,h,\eps}[\mu](\cdot, t)$ satisfies the following weak semiconcavity property
\be\label{weak_semiconcavity_mfg}
\left(\nabla_{x} v^{\rho,h,\eps}[\mu](y,t)-\nabla_{x} v^{\rho,h,\eps}[\mu](x,t)\right)\cdot (y-x) \leq c\left( |y-x|^{2} + \frac{\rho^2}{\eps^2}\right).
\ee
Using the previous ingredients, we can prove the following result.
%
\begin{proposition}\label{uniform_convergence_derivatives_mfg} Consider sequences $\rho_n$, $h_n$ and $\eps_n$ of positive numbers converging to $0$ and such that  $\frac{\rho_n^2}{h_n}\to 0$ and $\rho_{n} =o(\eps_n)$.  Then, for every sequence $\mu_n \in C([0,T]; \P_1(\RR^d))$ converging to $\mu$ we have that $v^{\rho_n,h_n,\eps_n}[\mu_n]$ and $\nabla_{x} v^{\rho_n,h_n,\eps_n}[\mu_n]$ converge to $v[\mu]$ and $\nabla_{x} v[\mu](y,t)$, respectively, uniformly over compact subsets of $\RR^d\times [0,T]$. 
\end{proposition}
\begin{proof} The assertion on the convergence of $v^{\rho_n,h_n,\eps_n}[\mu_n]$ is a consequence of the uniform convergence over compact sets of $v^{\rho_n,h_n}[\mu_n]$ to $v[\mu]$ if $\frac{\rho_n^2}{h_n}\to 0$, which is a standard result proved with the theory developed in \cite{MR1115933}  (see e.g. \cite[Theorem 4.2]{MR3042570}). The argument to establish the uniform convergence of $\nabla_{x} v^{\rho_n,h_n,\eps_n}[\mu_n]$   is similar to the proof of \cite[Theorem 3.5]{MR3148086}. Namely, for all $n\in \NN$  and $x_n\to x$ and $t_n\to t$, and $y\neq x$ we have (for $n$ large enough)
$$
v^{\rho_n,h_n,\eps_n}[\mu_n](y,t_n)-v^{\rho_n,h_n,\eps_n}[\mu_n](x_n,t_n)- \nabla_{x} v^{\rho_n,h_n,\eps_n}[\mu_n](x_n,t_n)\cdot(y-x_n) \leq r_{1n} + r_{2n},
$$
where 
$$\ba{l} r_{1,n}:= \int_{0}^{\frac{\rho_n}{\eps_n|y-x_n|}}\left[\nabla_{x} v^{\rho_n,h_n,\eps_n}[\mu_n](x_n+\tau (y-x_n),t_n)-\nabla_{x} v^{\rho_n,h_n,\eps_n}[\mu_n](x_n,t_n)\right]\cdot(y-x_n)\dd \tau, \\[6pt]
 r_{2,n}:= \int_{\frac{\rho_n}{\eps_n|y-x_n|}}^{1}\left[\nabla_{x} v^{\rho_n,h_n,\eps_n}[\mu_n](x_n+\tau (y-x_n),t_n)-\nabla_{x} v^{\rho_n,h_n,\eps_n}[\mu_n](x_n,t_n)\right]\cdot(y-x_n)\dd \tau
\ea
$$
Since $\frac{\rho_n}{\eps_n} \to 0$, the uniform Lipschitz character of $v^{\rho_n,h_n,\eps_n}[\mu_n](\cdot,t)$, for $t\in [0,T]$, implies that $r_{1,n} \to 0$. On the other hand, by \eqref{weak_semiconcavity_mfg}, 
$$r_{2,n}\leq  \int_{\frac{\rho_n}{\eps_n|y-x_n|}}^{1} \frac{c}{\tau}\left(\tau^2|y-x_n|^2+\left(\frac{\rho_n}{\eps_n}\right)^2\right) \dd \tau \leq \int_{0}^{1} \tau \dd \tau \leq \frac{c}{2}|y-x_n|^2.
$$
By the uniform convergence of $v^{\rho_n, h_n,\eps_n}[\mu_n]$, we conclude that any limit point $p$ of $\nabla_{x} v^{\rho_n,h_n,\eps_n}[\mu_n](x_n,t_n)$ (there exists at least one because this sequence is uniformly bounded) must satisfy
$$
v[\mu](y,t) \leq v[\mu](x,t) + p\cdot(y-x)+ \frac{c}{2} |y-x|^2 \hspace{0.4cm} \forall \; y \in \RR^d, \; t\in [0,T],
$$
which implies that $p= \nabla_{x} v[\mu](x,t)$  by \cite[Proposition 3.3.1 and Proposition 3.1.5(c)]{CannSinesbook}. Thus, if for all $i=1, \hdots, d$ we denote by 
$$
b_{i}^{\sup}:= \limsup_{x'\to x, t'\to t, n\to \infty} \partial_{x_i} v^{\rho_n,h_n,\eps_n}[\mu_n](x',t'), \; \; \; b_{i}^{\inf}:= \liminf_{x'\to x, t'\to t, n\to \infty} \partial_{x_i} v^{\rho_n,h_n,\eps_n}[\mu_n](x',t')
$$
we deduce that $b_{i}^{\sup}=b_{i}^{\inf}=\partial_{x_i}v[\mu](x,t)$ and so the local uniform convergence of $\nabla_{x}v^{\rho_n,h_n,\eps_n}[\mu_n](\cdot,\cdot)$ to $\nabla_{x} v[\mu](\cdot,\cdot)$ follows (see e.g. \cite[Chapter V, Lemma 1.9]{BardiCapuzzo96}). 
\end{proof}

Suppose that $\rho_n$, $h_n$ and $\eps_n$ satisfy the conditions in Proposition \ref{uniform_convergence_derivatives_mfg}, denote by $m^{n} \in C([0,T];\P_1(\RR^d))$ the extension to $C([0,T];\P_1(\RR^d))$ of the solution of  \eqref{scheme_nonlinear_stochastic_case_II} computed with coefficients $b^{n}[\mu](x,t):= \nabla_{x} v^{\rho_n,h_n,\eps_n}[\mu](x,t)$ and $\sigma_{\ell}^n= \sigma e_{\ell}$ ($\ell=1,\hdots, d$).  Using \eqref{uniform_Lipschitz_property_v} for $v^{\rho,h,\eps}$, we have the existence of $C>0$, independent of $\mu$, such that $| D^2_{xx} v^{\rho_n,h_n,\eps_n}|_{\infty} \leq C/\eps_n$ (see e.g.  \cite[Section 3]{CS15}). Therefore, if $h_n=o(\eps_n^2)$ we can reproduce the argument in the proof of Theorem \ref{convergencia_under_H} to obtain the following result.
\begin{proposition}\label{convergence_mgf_scheme} Under the above assumptions every limit point $m \in C([0,T]; \P_1(\RR^d))$ of $m^n$ {\rm(}there exists at least one{\rm)} solves \eqref{fpk_mfg_implicit}.
\end{proposition}
\begin{remark}\label{uniqueness_mfg}{\rm(i)} If $F$ and $G$ satisfy the following monotonicity conditions
$$\ba{l}
\int_{\RR^d} \left[ F(x,m_1) - F(x,m_2\right] \dd (m_1-m_2)(x) > 0 \hspace{0.4cm} \forall \; m_1, m_2 \in \P_1(\RR^d), \; m_1 \neq m_2,\\[6pt]
\int_{\RR^d} \left[ G(x,m_1) - G(x,m_2\right] \dd (m_1-m_2)(x) \geq 0 \hspace{0.4cm} \forall \; m_1, m_2 \in \P_1(\RR^d),
\ea
$$
then system \eqref{MFGimplicit} admits a unique solution $(v,m)$ {\rm(}see \cite{LasryLions07}{\rm)}. In this case the entire sequence $m^n$ in Proposition  \ref{convergence_mgf_scheme} converges to $m$. \smallskip\\
{\rm(ii)} In the articles \cite{MR3148086,CS15} a very similar scheme is proposed for degenerate  MFG systems  when  $m_0$ is absolutely continuous, with a compact support and with an essentially bounded density.  In those frameworks, the velocity field $b[\mu](x,t)$ is only defined for a.e. $x\in \RR^d$. Therefore {\rm(}see Remark \ref{remark_approximation_general_coefficients} {\rm (ii))}, the proposed scheme discretizes the density of $m$ for which an $L^\infty$ bound is proved if $d=1$. Moreover,  the authors show the $L^1$ convergence of the approximations of the velocity field, which is weaker than the result in Proposition \ref{uniform_convergence_derivatives_mfg}. On the other hand, when $d=1$, uniform bounds in $L^\infty$ are shown for the approximated densities,  which allows them to prove, in these degenerate  cases, a version of Proposition \ref{convergence_mgf_scheme}  in the one dimensional case. In their entire analysis, the extra assumptions on $m_0$ play an important role. 
\end{remark}
\subsubsection{Numerical test} \label{Testnonlinearimp}
We consider the MFG system \eqref{MFGimplicit} in dimension $d=r=1$ on the  space-time domain $\OO \times [0,T]:=[-3,3]\times [0,5]$,   $\sigma=0.01$ and  with   running  and terminal costs given respectively by
\be\label{eq:FeV2}\ba{c}
F(x,m):=d(x,\mathcal{P})^2V_{\delta}(x,m) \; \; \; \; G(x,m):=F(x,m),
\ea
\ee
where  
$$\ba{c}
V_{\delta}(x,m):= (\phi_{\delta} \ast \left(\phi_{\delta}  \ast m \right))(x)  \; \; \; \mbox{with } \; \; \; \phi_{\delta}(x):= \frac{1}{\delta \sqrt{2\pi}}e^{\frac{-x^2}{2\delta^2}},\ea$$
and $d(\cdot,\mathcal{P})$  denotes the distance to the set $\mathcal{P}:=[-2,-2.5]\cup[1,1.5]$. We choose as initial distribution 
$$\bar{m}_0(x)= \frac{\nu(x)}{ \int_{\OO}\nu(y) \dd y}  \mathbb{I}_{\OO}(x)  \hspace{0.3cm} \mbox{with } \; \;  \nu(x):=e^{-x^2/ 0.2}.$$
By formula \eqref{stochastic_optimal_control_problem_mfg} the interpretation in this setting is that agents want to reach  the meeting areas, defined by the set $\P$, without spending to much effort (modeled by the $|\alpha|^2$ term in \eqref{stochastic_optimal_control_problem_mfg}), and to avoid congestion, modeled by the coupling terms $F$ and $G$. Once the players reach the  meeting areas  they have not incentives to leave and they remain in $\P$.


We heuristically solve  the implicit scheme \eqref{scheme_nonlinear_stochastic_case_II}  using the learning procedure  proposed in \cite{CH17}  (analyzed at the continuous level). More precisely, given the discretization parameters $\rho$, $h$ and $\eps$ and an initial guess $m^0$ for the solution of \eqref{scheme_nonlinear_stochastic_case_II}, we compute $v^0$ by solving backwards \eqref{semi_lagrangian_scheme_v_mu} with $\mu=m^0$. The new iterate $m^1$ is computed using scheme \eqref{scheme_linear_stochastic_case} with 
$$
\Phi^{\pm}_{j,k}= x_{j}-h \tilde{\nabla} (v^0)^\eps_{j,k} \pm \sqrt{h}\sigma,
$$
where $\tilde{\nabla}  (v^0)^\eps_{j,k}$ is an approximation of $\nabla_{x} v^{\rho,h,\eps}[m^0](x_j,t_k)$. Then, given $m^p$ ($p\geq 1$) we compute $v^p$ by solving backwards \eqref{semi_lagrangian_scheme_v_mu} with $\mu= \frac{1}{p+1} \sum_{p'=0}^p m^{p'}$ and define $m^{p+1}$  using \eqref{scheme_linear_stochastic_case} with 
$$
\Phi^{\pm}_{j,k}= x_{j}-h \tilde{\nabla} (v^p)^\eps_{j,k} \pm \sqrt{h}\sigma,
$$
where 
$\tilde{\nabla}  (v^p)^\eps_{j,k}$ is an approximation of $\nabla_{x} v^{\rho,h,\eps}[m^p](x_j,t_k)$. We continue with these iterations until the difference between $m^p$ and $m^{p+1}$  is less than $0.01$ in the discrete infinity norm.

\begin{remark} Numerically, this heuristic performs rather well. The proof of convergence of this algorithm is not analyzed in this paper and it is postponed to a future work. One could expect that the arguments in \cite{CH17} apply to a discrete time, discrete space MFG {\rm(}see \cite{gomes10}{\rm)}. The main issue with the approximation   \eqref{scheme_nonlinear_stochastic_case_II} is that it does not correspond exactly to a discrete MFG because the distribution of the players does not evolve according to  the discrete optimal controls of the typical players {\rm(}computed as the optimizers of the r.h.s. of \eqref{semi_lagrangian_scheme_v_mu}{\rm)}, but with they evolve according to their approximations  $\nabla_{x} v^{\rho,h,\eps}[m^p](x_j,t_k)$. 
\end{remark}

The numerical approximation of the density ${\bf{m}}^{\rho,h,\eps}$ for $\rho=0.02$,  $h=\rho$, $\eps=0.15$ and  $\delta=0.02$  is depicted in  Figure \ref{Testimplicit}. 
In  Figure \ref{Testimplicit2d}, we plot the densities ${\bf{m}}^{\rho,h,\eps}$ at times $t=0$, $0.6$ and $5$. 
We observe that the density of agents  divides into three groups. The largest  one moves towards the right  meeting area which is the closest one. The second largest group moves  towards  the left area. The third and smallest group waits before moving towards the meeting area.  We note that in  this equilibrium,  the agents somehow take rational decisions based on their aversion to crowed places out of the meeting zones.
\begin{figure}[htp]
\begin{center}
       \includegraphics[width=3in]{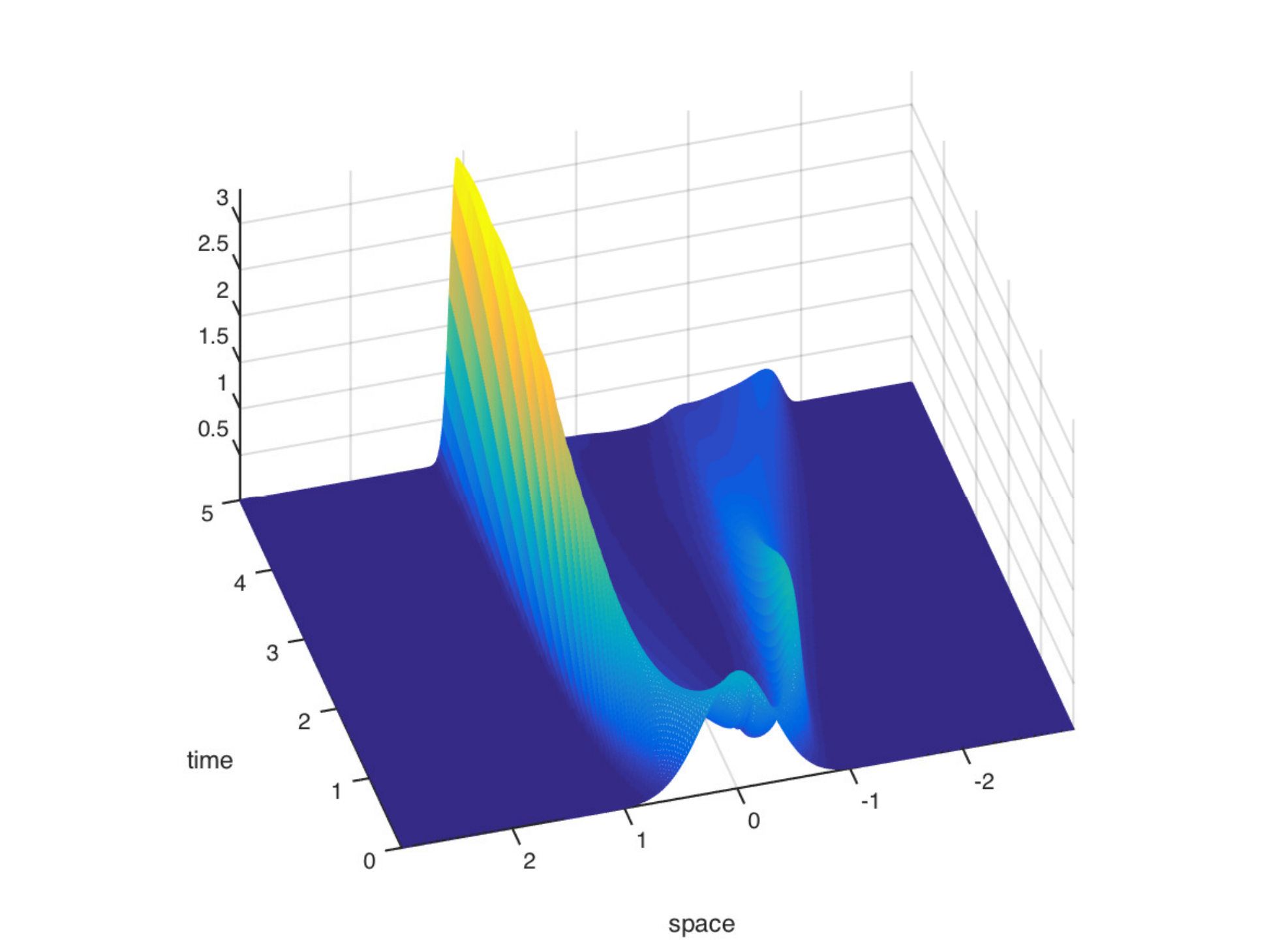}
        \includegraphics[width=3in]{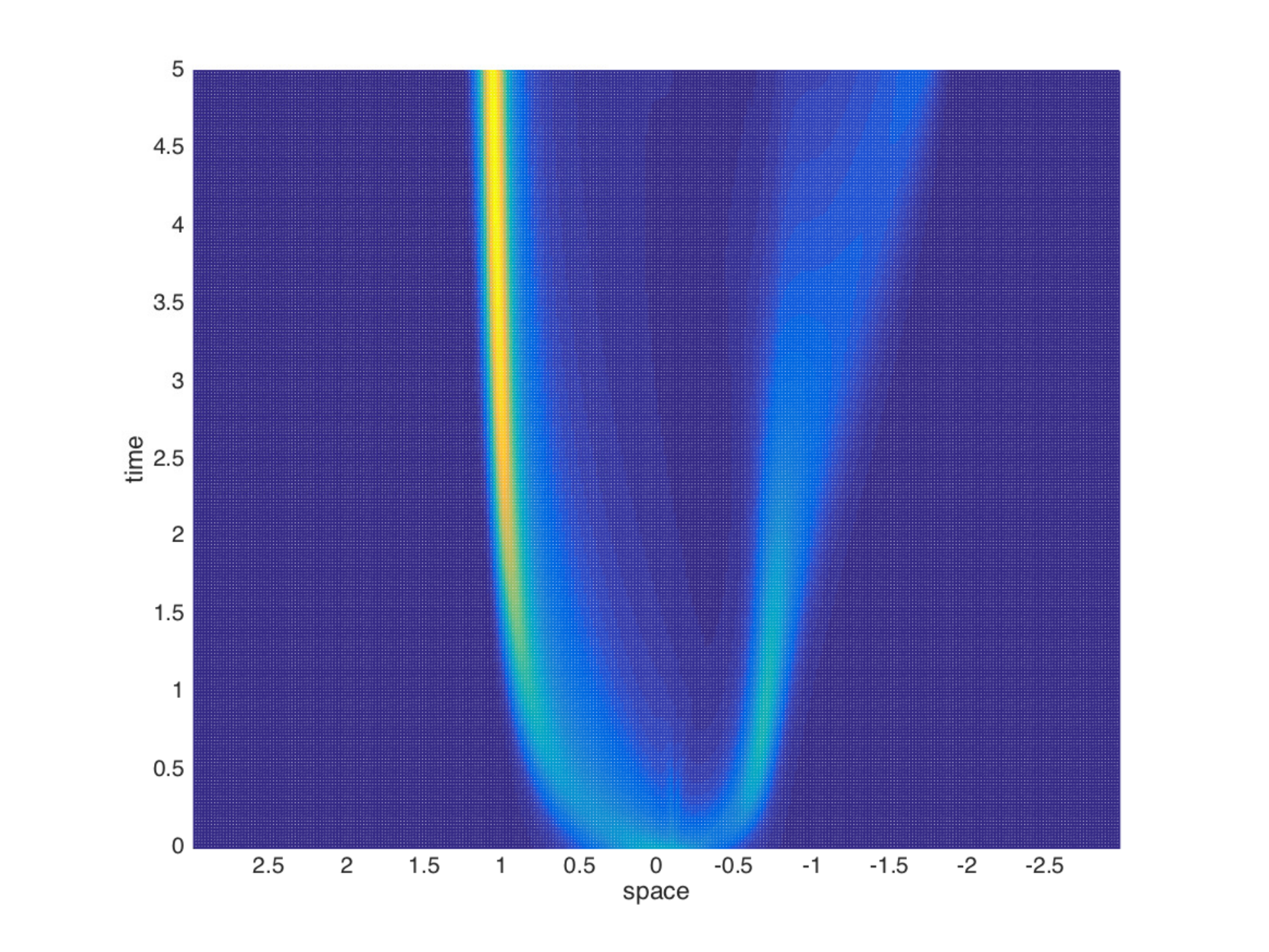}
\caption{Test \ref{Testnonlinearimp}:  3D and 2D view in the $(x,t)$ domain of the evolution of the density of agents.
\label{Testimplicit}}
\end{center}
\end{figure}
\begin{figure}[htp]
\begin{center}
 \includegraphics[width=3in]{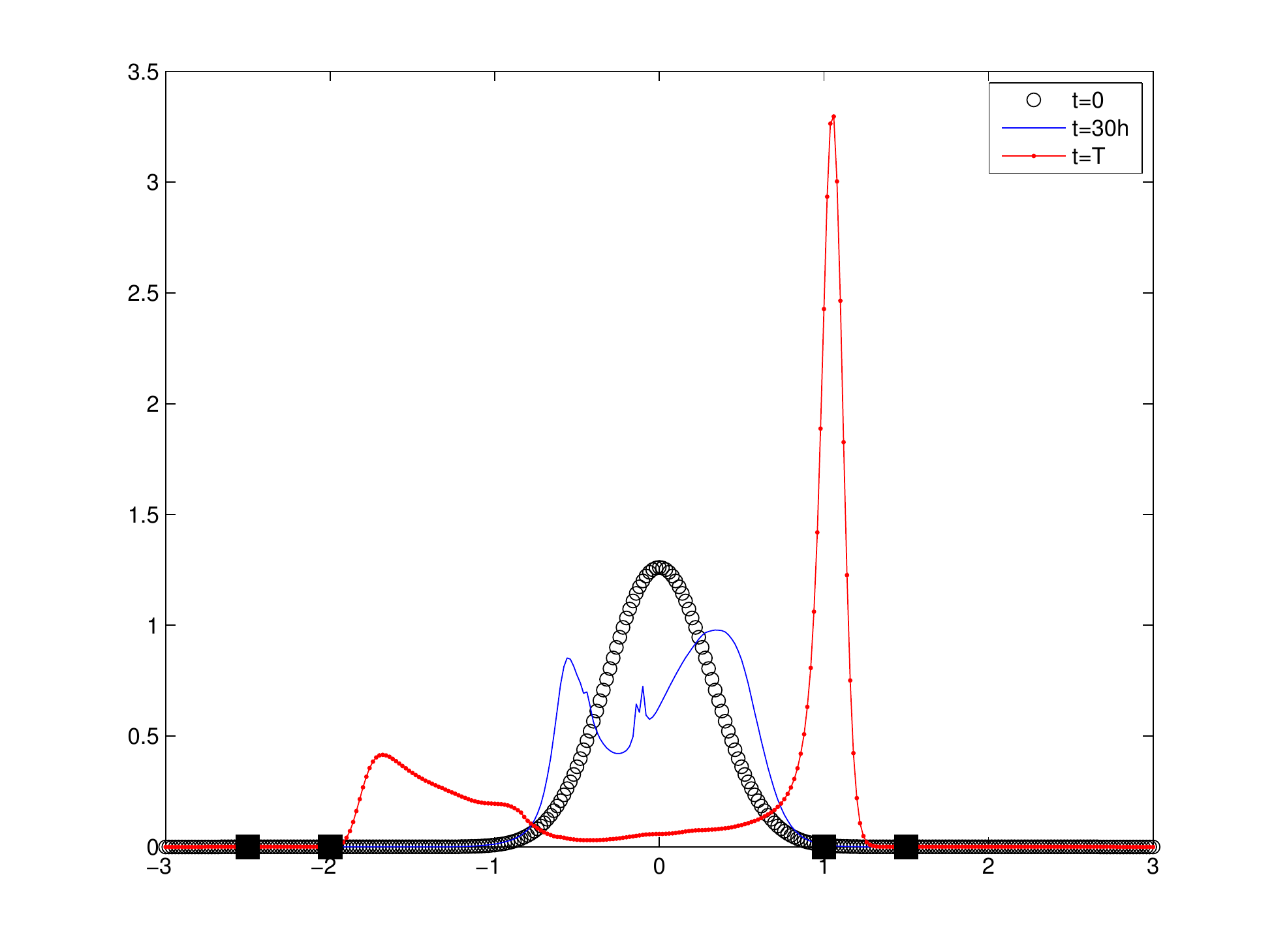}
\caption{ Test \ref{Testnonlinearimp}: Densities at times $t=0$, $0.6$ and $5$ (black squares on the $x$ axis  represent the boundary of  the ``meeting areas'').
\label{Testimplicit2d}}
\end{center}
\end{figure}
\subsection{A non-linear Hughes type  explicit model } In this section we consider the FPK equation
\be\label{fpk_hughes_explicit}
\partial_{t} m  -\frac{\sigma^2}{2}\Delta m   -\mbox{div} \big( \nabla v[m] m \big)  = 0 \; \; \; \mbox{in } \; \RR^d\times (0,T), \hspace{0.3cm} m(0)=\bar{m}_0,
\ee
where $v: C([0,T];\P_{1}(\RR^d)) \times \RR^d \times [0,T] \to \RR$ is given by 
\be\label{definition_b_explicit_hughes}
v[m](x,t):= \inf_{\alpha} \EE\left(\int_{t}^{T} \left[\half |\alpha(s)|^2 + F(X^{x,t,\alpha}(s), m(t)) \right] \dd s + G(X^{x,t,\alpha}(T),m(t))\right),
\ee 
and  the processes $\alpha$ and $X^{x,t,\alpha}$ are as in Section \ref{subsection_mfg}. We also assume that $F$ and $G$ satisfy \eqref{assumptions_F_G}. 

Note that the main difference with the MFG model considered in Section  \ref{subsection_mfg} is that the optimal control problem solved by an agent located at point $x$ at time $t$ depends on the global distribution $m$ of the agents {\it only} through its value at time $t$. In this sense, agents do not forecast, or in other words, no learning procedure has been adopted by the population of agents regarding their future behavior (see \cite{CH17} for the analysis of the {\it fictitious play} procedure in MFGs which can explain the formation of the equilibria).  This model is a variation of the one introduced by Hughes in \cite{di2011hughes} where the optimal control problem solved by the typical player is stationary of minimum time type.  In terms of PDEs, at each time $t\in (0,T)$ we consider the HJB equation 
\be\ba{rcl}\label{HJBmu_explicit} 
-\partial_{s} u(x,s)  -  \frac{\sigma^2}{2} \Delta u(x,s)    + \frac{1}{2}|\nabla u(x,s) |^2  &=& F(x, m(t)) \;  \;    \hbox{in } \RR^{d}\times (t,T), \\[6pt]
v(x,T)&=& G(x, m(t))\; \;   \mbox{for } x \in \RR^{d},
\ea\ee
which admits a classical solution $u[m(t)]$. We have that $v[m](x,t)= u[m(t)](x,t)$. By the continuity of $F$ and $G$, assumption \eqref{assumptions_F_G} and the representation formula \eqref{definition_b_explicit_hughes},  we have that $v$ is continuous. This can also be seen as a consequence of the stability of viscosity solutions with respect to continuous parameter perturbations (for equation \eqref{HJBmu_explicit} the parameter is $m(t)$). Moreover, as in the case of MFG,  assumption  \eqref{assumptions_F_G} implies that
\be\label{uniform_Lipschitz_property_v_hughes}\sup_{ t\in [0,T], \; m \in C([0,T];\P_{1}(\RR^d)) }   \|\nabla_{x}v[m](\cdot,t)\|_{\infty}  < \infty,\ee
and that for all $t\in [0,T]$,  $v[m](\cdot, t)$ is semiconcave, with a semiconcavity constant which is independent of $(m,t)$. Using this property and arguing exactly as in Section \ref{subsection_mfg} we obtain that $(m,x,t) \in C([0,T]; \P_1(\RR^d)) \times \RR^d \times [0,T] \to \nabla_{x}v[m](x,t) \in \RR^{d}$ is continuous and so Theorem \ref{convergencia_1} gives the following result.
\begin{proposition} Equation \eqref{fpk_hughes_explicit} admits at least one solution. 
\end{proposition}

As in the case of MFGs, in practice we do not known explicitly the velocity vector field $-\nabla_xv[m](x,t)$ and so we have to approximate it.  We consider the following approximation: given $\rho>0$, $h=T/N>0$, with $N\in \mathbb{N}$,  $\mu \in C([0,T];\P_1(\RR^d))$ and $k=0,\hdots, N-1$, we define\small
\be\label{semi_lagrangian_scheme_v_mu_hughes_case}
\ba{rcl} 
v^{\rho,h}_{i,k'}[\mu(t_k)]&=& \inf_{\alpha \in \RR^d} \left\{ \frac{h}{2}|\alpha|^{2} + \frac{1}{2d}\sum_{\ell=1}^{d}\left(I[v^{\rho,h}_{\cdot,k'+1}[\mu(t_k)]](x_{i} + h \alpha +\sigma \sqrt{hd}e_{\ell})\right. \right.\\[6pt]
\; &\; &\left.\left. \hspace{4cm} +I[v^{\rho,h}_{\cdot,k'+1}[\mu(t_k)]](x_{i} + h \alpha-\sigma \sqrt{hd}e_{\ell}) \right)\right\}\\[6pt]
\; & \; & \hspace{4cm}+ hF(x_i,\mu(t_k)) \hspace{0.4cm} \forall \; k'=k, \hdots, N-1,\\[6pt]
v^{\rho,h}_{i,N}[\mu(t_k)]&=& G(x_i,\mu(t_k)).
\ea
\ee \normalsize
We also define $v^{\rho,h}: C([0,T]; \P_{1}(\RR^d)) \times \RR^d \times [0,T] \to \RR$ by  
$$
v^{\rho,h}[\mu](x,t)=  I[v^{\rho,h}_{\cdot,k}[\mu(t_k)]](x,t_k) \hspace{0.5cm} \mbox{if } \; t\in [t_{k}, t_{k+1}[.
$$
Comparing with  \eqref{semi_lagrangian_scheme_v_mu}, where given $\mu \in C([0,T]; \P_1(\RR^d))$ the scheme discretizes only equation \eqref{HJBmu} (with $m$ replaced by $\mu$),  \eqref{semi_lagrangian_scheme_v_mu_hughes_case} discretizes the PDEs \eqref{HJBmu_explicit}  for each $t=t_k$ ($k=0,\hdots, N-1$). As in the case of MFGs, given $\eps>0$ and $\phi \in C^{\infty}(\RR^d)$, non-negative and such that $\int_{\RR^{d}} \phi(x)\dd x=1$, we define $v^{\rho,h,\eps}: C([0,T]; \P_{1}(\RR^d)) \times \RR^d \times [0,T] \to \RR$ by 
$$
v^{\rho,h,\eps}[\mu](\cdot, t):= \phi_{\eps} \ast v^{\rho,h}[\mu](\cdot,t) \hspace{0.3cm} \forall \; t\in [0,T],
$$
where $\phi_\eps(x):= \frac{1}{\eps^d} \phi(x/\eps)$.  By assumption \eqref{assumptions_F_G},  the bound \eqref{uniform_Lipschitz_property_v} and the semiconcavity property \eqref{weak_semiconcavity_mfg} remain valid in this context. Now, let  $\rho_n$, $h_n$ and $\eps_n$ satisfy the conditions in Proposition \ref{uniform_convergence_derivatives_mfg} and let $m^{n} \in C([0,T];\P_1(\RR^d))$  be the extension to $C([0,T];\P_1(\RR^d))$ of the solution to  \eqref{scheme_nonlinear_stochastic_case_II} computed with coefficients $b^{n}[\mu](x,t):= \nabla_{x} v^{\rho_n,h_n,\eps_n}[\mu](x,t)$ and $\sigma_{\ell}^n= \sigma e_{\ell}$ ($\ell=1,\hdots, d$).  As before, using that $\nabla_{x} v^{\rho_n,h_n,\eps_n}[m_n](\cdot,t)$ is uniformly bounded in $t$ and $n$, we have that $m^n$ has at least one limit point $m\in C([0,T];\P_1(\RR^d))$. Moreover, reasoning as in the proof of \cite[Theorem 3.3]{MR3148086}, for each fixed $t\in [0,T]$ we have that $v^{\rho_n,h_n,\eps_n}[m_n](\cdot, t)\to u[m(t)](\cdot, t)=v[m](\cdot, t) $ and so, by  \eqref{weak_semiconcavity_mfg} and the proof of Proposition \ref{uniform_convergence_derivatives_mfg}, we have that $\nabla_x v^{\rho_n,h_n,\eps_n}[m_n](\cdot, t)\to \nabla_xv[m](\cdot, t)$ uniformly on compact sets of $\RR^d$.  As in the case of MFGs, we have the existence of  a constant $C>0$, independent of $\mu$, such that $| D^2_{xx} v^{\rho_n,h_n,\eps_n}|_{\infty} \leq C/\eps_n$. Therefore,    we can argue exactly as in the proof  of Theorem \ref{convergencia_under_H} to obtain the following result.
\begin{proposition} Assume that $\rho_n^2=o(h_n)$ and $h_n=o(\eps_n^2) $. Then,  every limit point $m \in C([0,T]; \P_1(\RR^d))$ of $m^n$ {\rm(}there exists at least one{\rm)} solves \eqref{fpk_hughes_explicit}.

\end{proposition}
\subsubsection{ Numerical test } \label{testexplicitmodel}
For the sake of comparison,  we consider here the same framework than the one  in Subsection \ref{Testnonlinearimp}, i.e. we take $d=r=1$, we work on the domain $\OO\times [0,T]=[-3,3]\times [0,5]$ and we impose an homogenous Neumann boundary condition on the FPK equation \eqref{fpk_hughes_explicit}. The functions $F$ and $G$ are also as in the previous test, as well as the initial distribution $\bar{m}_0$ of the agents.

We  proceed iteratively in the following way: given the discrete measure $m_{k}^{\rho,h,\eps}$  at time $t_k$ ($k=0,\hdots, N-1$), we  compute  at each space grid point $j$ the discrete value function $v_{j,k}$ by using \eqref{semi_lagrangian_scheme_v_mu_hughes_case} with $\mu(t_k)$ replaced by $m_{k}^{\rho,h,\eps}$. We  regularize the interpolated function $I[v_{\cdot,k}]$ by using   a discrete  space convolution with a mollifier 
$\phi_{\eps}$.  We denote by  $ \tilde{\nabla} v^\eps_{j,k} $  the approximation of its spatial gradient at $x_j$.  Then we calculate $m_{k+1}^{\rho,h,\eps}$  with scheme \eqref{scheme_nonlinear_stochastic_case_II} by approximating the discrete trajectories by 
$$
\Phi^{\pm}_{j,k}= x_{j}-h \tilde{\nabla} v^\eps_{j,k}\pm \sqrt{h}\sigma,
$$
and we iterate the process until $k=N-1$.  Note that, by construction, the scheme is explicit in time. 

 The  approximation of the density  evolution in the $(x,t)$ domain,  computed with $\rho=0.02$, $h=\rho$, $\eps=0.15$ and $\delta=0.01$,  is shown in  Figure \ref{Testexplicit}.  In  Figure \ref{Testexplicit2d}, we plot the approximated density at times $t=0$, $0.6$ and $5$. 
 We observe that the initial density $\bar{m}_0$  divides into two parts. The first one quickly reaches the meeting  area on the right   and once there it stops and begins to accumulate in this zone. The  second part of the density moves in the opposite direction trying to reach the left  meeting area.  In contrast to the presented MFG model, in this model the agents make their decisions based only in the current global configuration. As a consequence,  we observe faster and higher accumulation of agents in the meeting zones.
  
  
%
    
\begin{figure}[htp]
\begin{center}
 \includegraphics[width=3in]{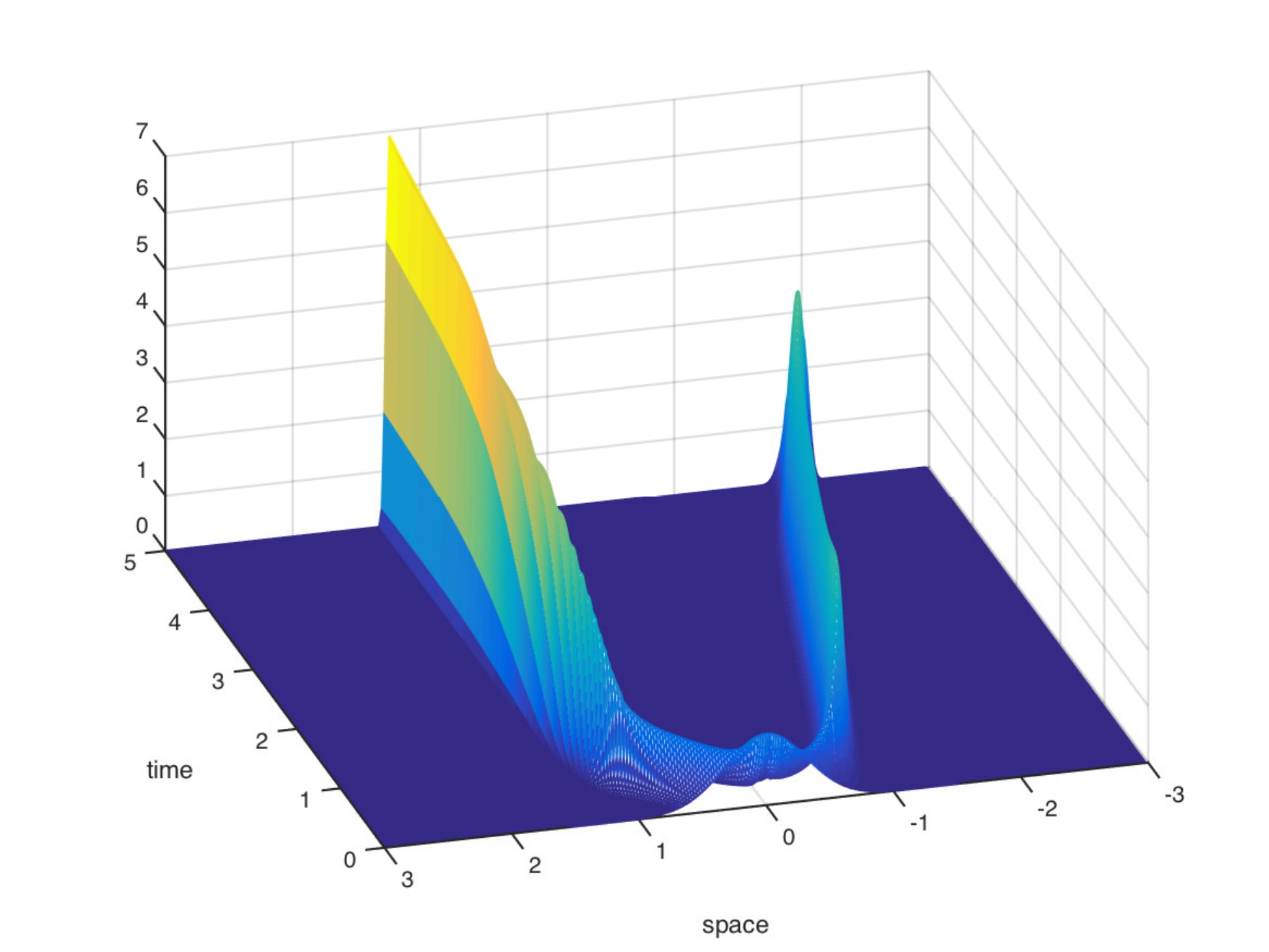}
   \includegraphics[width=3in]{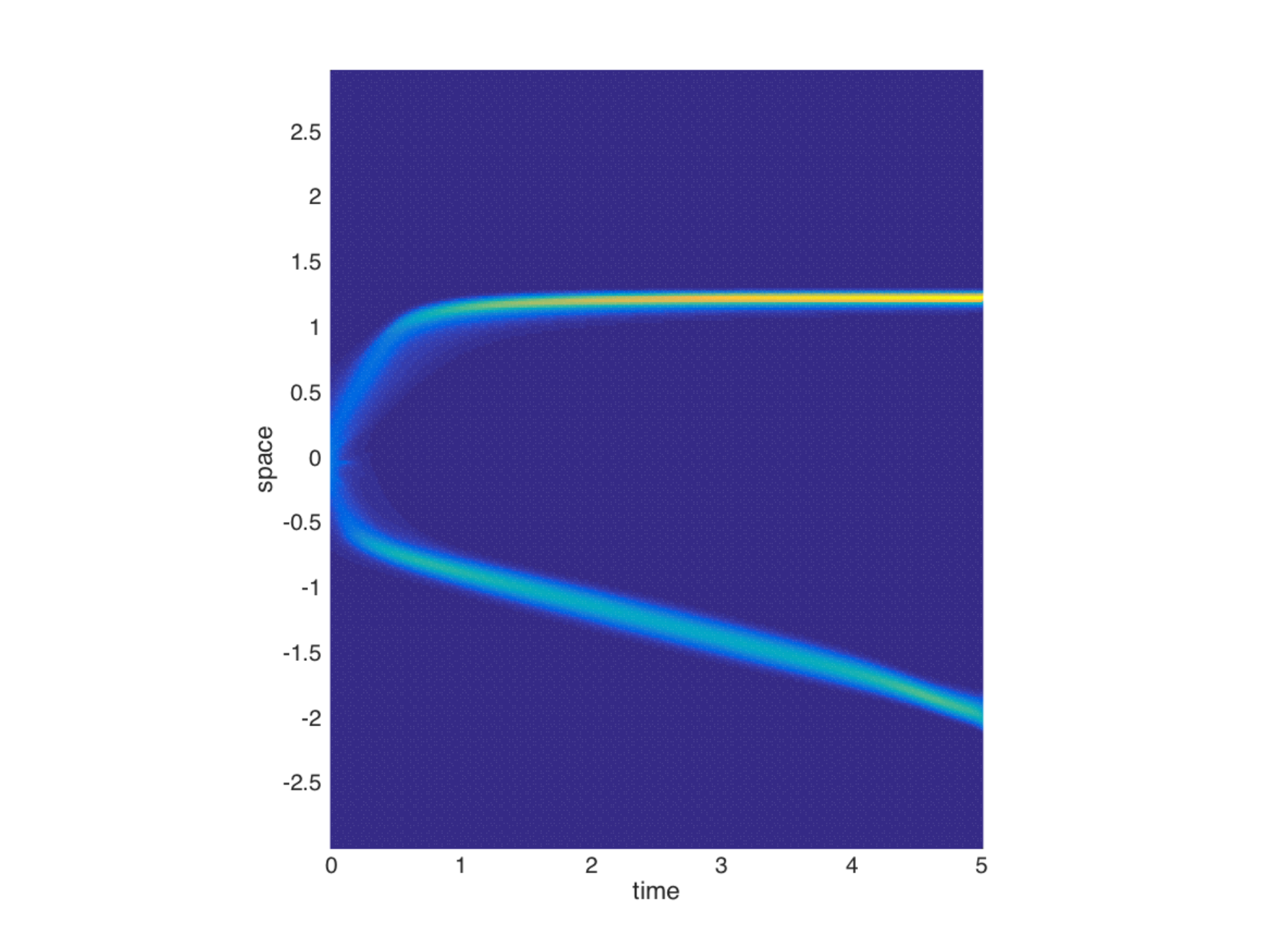}
\caption{Test \ref{testexplicitmodel}:  3D and 2D view in the $(x,t)$ domain of the evolution of the density of agents.
\label{Testexplicit}}
\end{center}
\end{figure}
\begin{figure}[htp]
\begin{center}
  \includegraphics[width=3in]{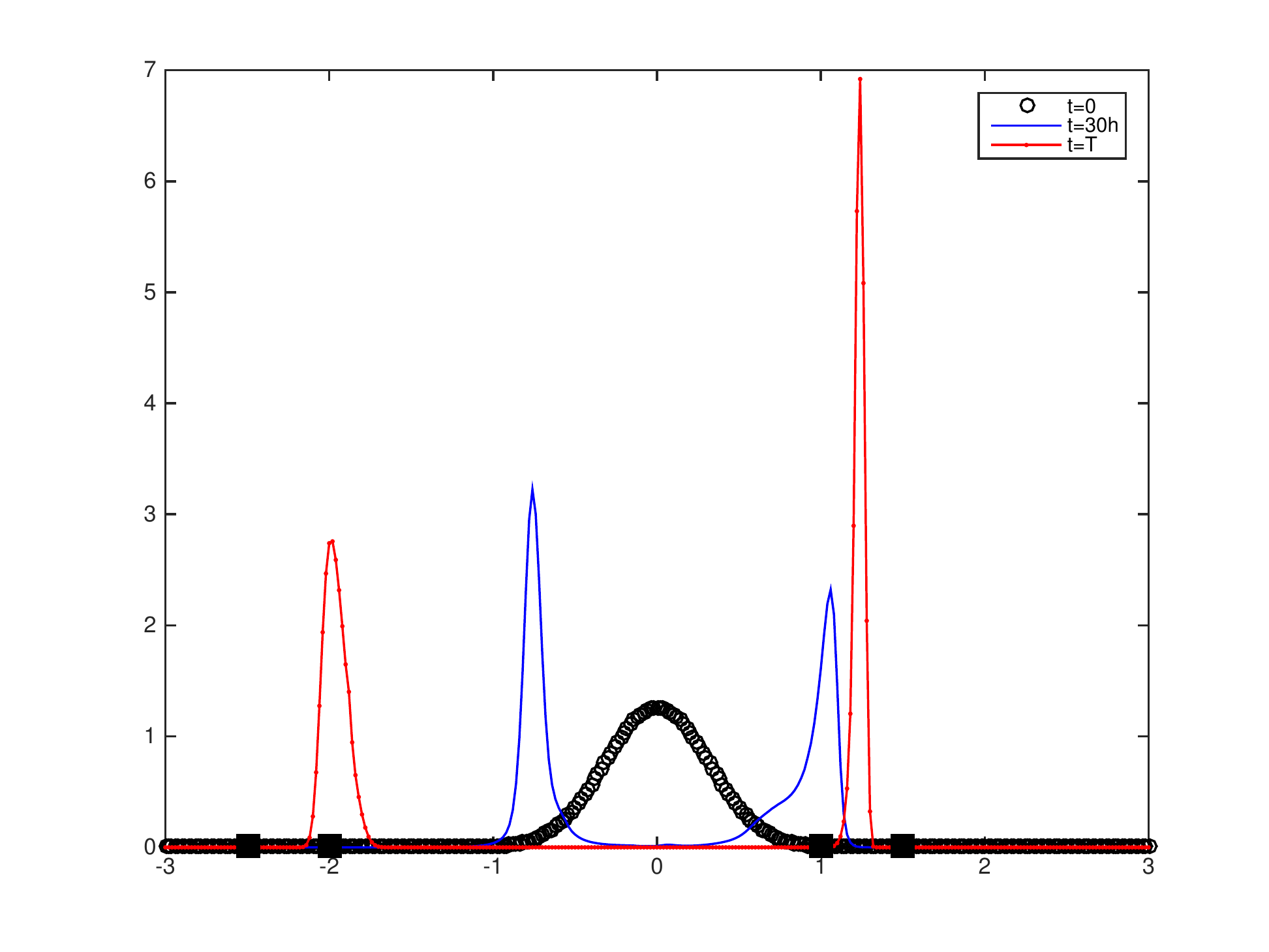}
\caption{ Test \ref{testexplicitmodel}: density  of agents at times $t=0$, $0.6$ and $5$ (black squares on the $x$ axis represents the boundary of  the `meeting areas').}
\label{Testexplicit2d}
\end{center}
\end{figure}

%
%
%
\bibliographystyle{plain}
\bibliography{bibFP}
\end{document}